\documentclass[a4paper, oneside, 11pt]{amsart}

\usepackage[T1]{fontenc}              
\usepackage[latin1]{inputenc}    
\usepackage{geometry}

\usepackage{amsmath, amsfonts,amssymb} 
\usepackage{mathtools}
\usepackage{bbm}
\usepackage{amsthm}     
\usepackage[capitalize]{cleveref}
\usepackage{graphicx}
\usepackage{subfig}     
\usepackage{listings}
\usepackage{enumerate}
\usepackage{tikz-cd}
\usepackage{rotating}
\usepackage{stmaryrd}
\usepackage{soul} 

\usepackage{todonotes}

\newtheorem{lem}{Lemma}[section]
\newtheorem{prop}[lem]{Proposition}

\newtheorem{thm}[lem]{Theorem}

\theoremstyle{definition}
\newtheorem{defin}[lem]{Definition}
\newtheorem{remark}[lem]{Remark}

\newtheorem{example}[lem]{Example}

\DeclareMathOperator{\rad}{rad}

\DeclareMathOperator{\topM}{top}
\DeclareMathOperator{\socM}{soc}
\DeclareMathOperator{\gldim}{gldim}

\DeclareMathOperator{\add}{add}
\DeclareMathOperator{\Ext}{Ext}
\DeclareMathOperator{\Gr}{Gr}
\DeclareMathOperator{\gr}{gr}
\DeclareMathOperator{\stgrmodu}{\setul{0.6ex}{0.1ex}
\textup{\ul{gr}}}

\DeclareMathOperator{\Hom}{Hom}
\DeclareMathOperator{\RHom}{\mathbb{R}Hom}

\DeclareMathOperator{\Tor}{Tor}

\DeclareMathOperator{\modu}{mod}

\DeclareMathOperator{\stTor}{\setul{0.6ex}{0.1ex}
\textup{\ul{Tor}}}
\DeclareMathOperator{\stExt}{\setul{0.6ex}{0.1ex}
\textup{\ul{Ext}}}

\DeclareMathOperator{\stmodu}{\setul{0.6ex}{0.1ex}
\textup{\ul{mod}}}

\DeclareMathOperator{\Hm}{H}
\DeclareMathOperator{\HH}{HH}
\DeclareMathOperator{\stHH}{\setul{0.6ex}{0.1ex}
\textup{\ul{HH}}}
\DeclareMathOperator{\HC}{HC}

\DeclareMathOperator{\op}{op}

\DeclareMathOperator{\D}{\mathcal{D}}

\DeclareMathOperator{\U}{\mathcal{U}}
\DeclareMathOperator{\C}{C}
\DeclareMathOperator{\K}{K}

\DeclareMathOperator{\tr}{tr}
\DeclareMathOperator{\Coeff}{Coeff}

\usepackage[foot]{amsaddr}

\title[\resizebox{5.25in}{!}{Hochschild (co)homology and cyclic homology via a graded Euler characteristic with applications to higher preprojective algebras}]{Hochschild (co)homology and cyclic homology via a graded Euler characteristic with applications to higher preprojective algebras}

\author{Jon Wallem Anundsen and Mads Hustad Sand\o y}

\begin{document}


\address{Department of mathematical sciences, NTNU, NO-7491 Trondheim, Norway}
\email{jon.w.anundsen@alumni.ntnu.no}
\email{mads.sandoy@ntnu.no}

\begin{abstract}
Computing the structure of the Hochschild (co)homology and the cyclic homology of an algebra can be hard work, but Etingof and Eu \cite{Etingof-Eu'06} showed that it can be done surprisingly easily for preprojective algebras of ADE Dynkin type, at least if one only wants to know the graded vector space structure of each Hochschild cohomology group. 
Their method is based on exploiting strong structural features of such a preprojective algebra via a graded Euler characteristic that can computed using the algebra's graded Cartan matrix. 
In this paper, we present a generalization of the method used by Etingof and Eu to higher preprojective algebras. 
We also apply our generalization to the higher preprojective algebras of the $2$-representation finite algebras that arise as tensor products of representation finite hereditary algebras of type $\mathbb{A}$. 
For this, it turns out to be enough to know the graded vector space structure of the center and the zeroth Hochschild homology to be able to deduce the graded vector space structure of the Hochschild (co)homology and the cyclic homology in all other degrees.
\end{abstract}

\maketitle

\tableofcontents

\section{Introduction}
Hochschild (co)homology and cyclic homology are all notorious for how difficult they can be to compute for non-trivial examples. 
It is thus quite striking when a method allows all three of these forms of (co)homology to be computed surprisingly easily for otherwise seemingly complicated classes of algebras, as is the case in \cite{Etingof-Eu'06} and in the works \cite{Evans-Pugh'12, Morigi'22}: 
In particular, under the assumption that the base field has characteristic zero, the Hochschild (co)homology and cyclic homology was computed for preprojective algebras of ADE Dynkin algebras in \cite{Etingof-Eu'06}. 
Using essentially the same approach, \cite{Evans-Pugh'12} achieved similar results for a ``higher'' version of preprojective algebras of ADE Dynkin algebras in all types except type $\mathbb{A}$, which was instead later done by \cite{Morigi'22} up to assuming the veracity of two conjectures. 

The simplifications their method yields can be quite staggering. 
Indeed, in \cite{Etingof-Eu'06} the method allowed one to infer the Hochschild (co)homology and cyclic homology in all degrees from knowing the zeroth Hochschild homology.
In other words, it completely removed the need to do explicit computations using complexes!
Of course, for each of these classes of algebras, the various forms of (co)homology are periodic, and so even with a naive approach it would suffice to compute them in finitely many degrees. 
However, the simplifications afforded by the method cuts down the number of degrees one needs to compute dramatically: e.g.\ for \cite{Etingof-Eu'06} it goes from five to one. 

The goal of the present paper is to generalize the method or approach used in the aforementioned papers. 
However, to explain how and why we do this, we  first give a fairly ``handwavy'' review of the method itself: 
Let $\Lambda = \oplus_{i \geq 0} \Lambda_i$ be a positively graded finite dimensional $K$-algebra for a field $K$ of characteristic zero, and let $\overline{\HH}_i(\Lambda)$ and $\overline{\HC}_i(\Lambda)$ be its $i$th reduced Hochschild and cyclic homology, respectively.
Note that $\overline{\HH}_i(\Lambda) = \HH_i(\Lambda)$ holds for $i > 0$. 
In brief, the method begins by using that for any $n$ there is a short exact sequence
$$
0 \to \overline{\HC}_{n - 1}(\Lambda) \to \overline{\HH}_n(\Lambda) \to \overline{\HC}_n(\Lambda) \to 0
$$
by \cite[Theorem 4.1.13]{loday}, splicing these together to form the Connes long exact sequence
$$
0 \to \overline{\HH}_0(\Lambda) \to \overline{\HH}_1(\Lambda) \to \overline{\HH}_2(\Lambda) \to \overline{\HH}_3(\Lambda) \to \cdots,
$$
and then attempting to leverage the graded Euler characteristic
$$\sum_{i = 1}^{\infty} a_i x^i = \sum_{i = 1}^{\infty}\sum_{j = 0}^{\infty} (-1)^{i}\dim \overline{\HC}_{i}(\Lambda)_j x^j.$$

In the papers cited above, this is done by using that 
\begin{enumerate}[(A)]
    \item $\Lambda$ is positively graded such that $\Lambda = \oplus_{i\geq 0} \Lambda_i$ with $\Lambda_0$ semisimple. Hence, if $C_\Lambda(x)$ is the graded Cartan matrix of $\Lambda$, one has by \cite{Etingof-Ginzburg'07} that
$$\prod_{i = 1}^{\infty}(1 - x^i)^{-a_i} = \prod_{s = 1}^{\infty}\det C_\Lambda(x^s).$$
\end{enumerate}

However, this by itself is not sufficient, and so further strong properties of $\Lambda$ must be invoked.
In the case of \cite{Etingof-Eu'06, Evans-Pugh'12, Morigi'22}, these properties include that the algebras in question are almost Koszul in the sense of \cite{BBK02}: Roughly speaking, a finite dimensional algebra is almost Koszul if it is as close as possible to actually being Koszul (as in \cite{Priddy'70}, but see also \cite{BGS96}) while also satisfying that its simple modules are periodic. 
Indeed, let $\Lambda = \oplus_{i = 0}^{p}\Lambda_i$ be almost Koszul and recall that it then has a ``dual'' algebra 
$$\Lambda^! := \oplus_{i = 0}^{q}\Ext^{i}_{\gr \Lambda}(\Lambda_0, \Lambda_0 \langle i \rangle).$$
Moreover, recall that if $p,q \geq 2$ and $\Lambda$ is connected and basic, then both $\Lambda$ and $\Lambda^!$ are selfinjective and even Frobenius.
In this case, (A) is actually part of the assumption that $\Lambda$ is almost Koszul, but we also obtain the following: 

\begin{enumerate}[(A)]
    \setcounter{enumi}{1}
    \item $\Lambda$ has a minimal graded projective resolution $P^\bullet_{\Lambda^{e}} \twoheadrightarrow \Lambda$ of $\Lambda$ considered as a bimodule that is ``piece-wise linear'' by \cite{Yu'12} (and ``twisted periodic'' by \cite{Green-et-al-2}); and
    \item the graded Cartan matrices of $\Lambda$ and $\Lambda^!$ satisfy $$C_\Lambda(x) C_{\Lambda^!}(-x)^T = I + (-1)^{q} x^{p+q} P^{p+q}$$ for a permutation matrix $P$ corresponding to the product of the Nakayama permutation of $\Lambda$ and the inverse of the Nakayama permutation of $\Lambda^!$; see Proposition 3.14 of \cite{BBK02}.  
\end{enumerate}

Note that (C) allows one to reduce to  whichever one of $C_\Lambda(x)$ and $\C_{\Lambda^!}(x)$ is easier to work with, and it often results in significant simplifications. 

Let $\mu$ be the Nakayama automorphism of $\Lambda$ and recall that ${}_{1}\Lambda_{\mu^{-1}}$ is the bimodule obtained by twisting the right $\Lambda$-action of $\Lambda$ by $\mu^{-1}$, i.e.\ if $\lambda, \lambda'' \in \Lambda$ and $\lambda' \in {}_{1}\Lambda_{\mu^{-1}}$, then  $\lambda \cdot \lambda' \cdot \lambda'' = \lambda\lambda'\mu(\lambda'')$. 
With this in mind, we note that in all of \cite{Etingof-Eu'06, Evans-Pugh'12} and \cite{Morigi'22}, the algebras satisfy that
\begin{enumerate}[(D1)]
    \item $\Omega^{m+1}_{\Lambda^{e}} \Lambda$ and ${}_{1}\Lambda_{\mu^{-1}}$ are isomorphic as bimodules for some integer $m \geq 2$; and 
    \item $\mu$ has finite order. 
\end{enumerate} 
In the following example, we try to show how all of this comes together by sketching what is done in \cite{Etingof-Eu'06}. 
\begin{example}\label{iex: Etingof-Eu}
Let $\Lambda$ be an ADE Dynkin preprojective algebra; see  \cite{Gelfand-Ponomarev}, \cite{DR80}, but also note that e.g.\ \cite{Rin98} is a readable source discussing various equivalent definitions of preprojective algebras.
Then, by \cite{Schofield}, we have that $\Omega^3_{\Lambda^{e}}\Lambda \cong \Lambda_{\mu^{-1}}$ so that (D1) holds; and (D2) holds since it is well known that $\mu^2 = 1$.
This forces a certain self-duality of a portion of the Connes long exact sequence for $\Lambda$ via the formula $\HH_i (\Lambda) \cong D\HH_{5 - i} (\Lambda)$ that follows by also using that $\Lambda^{e} := \Lambda^{\op} \otimes \Lambda$ is Frobenius since $\Lambda$ is Frobenius. 
In fact, more can be said: Note that it is well known that $\Lambda^!$ is in this case isomorphic to the trivial extension $\Delta (A^!)$ for $A = KQ$ and $Q$ an orientation of an ADE Dynkin quiver of the same type as $\Lambda$; see e.g.\ \cite{BBK02}.
As a consequence, we have that
$$\Lambda^! = \oplus_{i = 0}^{2}\Ext^{i}_{\gr \Lambda}(\Lambda_0, \Lambda_0 \langle i \rangle),$$
i.e.\ $\Lambda^!$ is of highest degree $q = 2$. 

We now know that
$$C_{\Lambda^!}(x) = I + x (M + M^T) + x^2I$$
for $M$ the adjacency matrix of $A^!$. We can thus in principle use (A) and (C) in combination to compute the graded Euler characteristic of $\overline{\HC}_*(\Lambda)$.
For instance, one can find the values of $\det C_{\Lambda^!}(x)$ in \cite{Lusztig'83}.
In case the Dynkin type of $\Lambda$ is $\mathbb{A}_n$, we get that $\det C_{\Lambda^!}(x) = (1 - x^{n+1})^2$. 
Moreover, the matrix $P$ in (C) can similarly be determined by consulting a source such as \cite{Schofield} or \cite{Erdmann-Snashall'98}; but see also \cite{HI11b}.
If the Dynkin type is $\mathbb{A}_n$, one obtains 
\begin{align*}
\chi_{\overline{\HC}_*(\Lambda)}(x) 
& = 
-\frac{1}{1 - x^{2n+2}} \left( \sum_{i = 1}^{n}x^{2i} -x^{n+1} + nx^{2n+2} \right).
\end{align*}

We now let $h$ be the Coxeter number of the Dynkin type of $\Lambda$ --- e.g.\  $h = n + 1$ if the type is $\mathbb{A}_n$ --- and observe that $\Omega^{3}_{\Lambda^{e}}\Lambda \cong \Lambda_{\mu^{-1}}\langle h\rangle$ holds by \cite{Yu'12} as graded bimodules, where $h = h - 2 + 2$ is the sum of the highest degrees of $\Lambda$ and $\Lambda^!$.
Note that given a graded $\Lambda$-module $M$, we let $M \langle j \rangle$ denote the $j$\textit{-th graded shift of} $M$, i.e.\ $M \langle j \rangle$ has the same underlying $\Lambda$-module structure as $M$, but its grading is defined by setting $M\langle j \rangle_i := M_{i - j}$.
Hence, we obtain that (B), (D1) and (D2) in combination yield that $\Omega^{6}_{\Lambda^{e}}(\Lambda) \cong \Lambda\langle 2h \rangle$. Hence, we see that $\HH_{i + 6} (\Lambda) \cong \HH_i (\Lambda) \langle 2h \rangle$ for $i \geq 1$. 
Moreover, another fairly straightforward computation then shows that $\HH_i (\Lambda) \cong D\HH_{5 - i} (\Lambda)\langle 2h \rangle$ for $1 \leq i \leq 4$.

Consider now the following diagram obtained from the Connes long exact sequence. 

\[
\begin{tikzcd}
& 0 \dar & \\
0 \leq \deg \leq h - 2 & \overline{\HH}_0(\Lambda) \dar["B_0"] \rar[equal, shorten = 2mm, shift right] & C \dar["\sim" {rotate=90, anchor=north}] &\\
1 \leq \deg \leq h - 1  & \overline{\HH}_1(\Lambda) \dar["B_1"] \rar[equal, shorten = 2mm, shift right] & C \rar[phantom, "\oplus"] & X_1 \dar["\sim" {rotate=90, anchor=north}]\\
2 \leq \deg \leq h & \overline{\HH}_2(\Lambda) \dar["B_2"] \rar[equal, shorten = 2mm, shift right] & X_2 \dar["\sim" {rotate=90, anchor=north}] \rar[phantom, "\oplus"] & X_1\\
h \leq \deg \leq 2h - 2 & \overline{\HH}_3(\Lambda) \dar["B_3"] \rar[equal, shorten = 2mm, shift right] & D(X_2)\langle 2h \rangle \rar[phantom, "\oplus"] & D(X_1)\langle 2h \rangle \dar["\sim" {rotate=90, anchor=north}]\\
h + 1 \leq \deg \leq 2n - 1 & \overline{\HH}_4(\Lambda) \dar["B_4"] \rar[equal, shorten = 2mm, shift right]  & D(C)\langle 2h \rangle \dar["\sim" {rotate=90, anchor=north}] \rar[phantom, "\oplus"] & D(X_1)\langle 2h \rangle\\
h + 2 \leq \deg \leq 2h & \overline{\HH}_5(\Lambda) \dar["B_5"] \rar[equal, shorten = 2mm, shift right]  & D(C)\langle 2h \rangle \rar[phantom, "\oplus"] & L \langle 2h \rangle  \dar["\sim" {rotate=90, anchor=north}]\\
2h \leq \deg \leq 3h - 2 & \overline{\HH}_6(\Lambda)  \arrow[r, equal, shorten = 2mm, shift right]  & X_3 \rar[phantom, "\oplus"]  & L \langle 2h \rangle 
\end{tikzcd}
\]

The bounds on the degrees depicted in the left-most column come from the ``piece-wise linear'' property in (B), i.e.\ from \cite{Yu'12}. 
Using these bounds, we note that it is straightforward to read off that $X_1$ is trivial in degrees less than $2$ and in degrees greater than $h-1$, that $X_2$ is trivial except possibly in degree $h$, that $L\langle 2h\rangle$ is trivial except possibly in degree $2h$, and thus we deduce that $X_1, X_2 \cong D(X_{2})\langle 2h\rangle, D(X_1)\langle 2h\rangle$ and $L$ are non-trivial in disjoint sets of degrees. 
Hence, as long as one knows the Hilbert series of $C$ --- which we recall is $h_C(x) = \sum_{i\in \mathbb{N}} \dim_K (C_i) x^{i}$ for a positively graded $K$-vector space $C$--- one could easily obtain the Hilbert series of $X_1, X_2, D(X_1)\langle 2h\rangle$ and $L$ from the Euler characteristic we could compute using (A) and (C).

Since $\HH_0 (\Lambda) \cong \Lambda/[\Lambda, \Lambda]$ holds for any $K$-algebra, it seems reasonable to hope to be able to compute the structure of $C$ as a graded vector space. 
In fact, \cite{Malkin-Ostrik-Vybornov'06} shows that $\HH_0(\Lambda) \cong \Lambda/\rad \Lambda$ in case $\Lambda$ is a preprojective of an ADE Dynkin algebra. 
Hence, $C$ is trivial, implying that $B_4 = 0$ and thus also that $X_3$ is trivial.  
Consequently, one can compute $\HH_i(\Lambda)$ for all $i \geq 7$ by using that $\HH_{i + 6} (\Lambda) \cong \HH_{i} (\Lambda)\langle 2h \rangle$ holds for $i \geq 1$ as mentioned above. 

By using (D1)-(D2) and the fact that $\Lambda^{e}$ is Frobenius, we can rewrite the computations above in terms of Hochschild cohomology: e.g.\ one obtains that $\HH^{i}(\Lambda) \cong \HH_{2 - i}(\Lambda) \langle - 1 \rangle$ holds for $i = 1$ while $\HH^{i}(\Lambda) \cong \HH_{8 - i}(\Lambda) \langle -2h - 1 \rangle$ holds for $2 \leq i \leq 7$.

\subsection{Our results}
We generalize the method used in \cite{Etingof-Eu'06, Evans-Pugh'12} and \cite{Morigi'22} by generalizing the assumptions necessary for the aforementioned versions of (A), (B), and (C) to hold. 
In particular, we are able to weaken the assumption that the algebra is graded in such a way that the degree zero part is semisimple. 
For (A), this takes the following form. 
\begin{thm}[See \cref{thm: etingof-ginzburg}]\label{ithm: etingof-ginzburg}
Let $\Lambda = \oplus_{i \geq 0} \Lambda_i$ be a positively graded $K$-algebra satisfying the following assumptions:
\begin{enumerate}[(i)]
    \item $\Lambda = \oplus_{i \geq 0}\Lambda_i$ is locally finite dimensional, i.e.\ each homogeneous component $\Lambda_i$ of $\Lambda$ is finite dimensional.
    \item The homogeneous component of $\Lambda$ of degree zero is given by a finite acyclic quiver with admissible relations.
\end{enumerate}
Also let $\chi_{\overline{\HC}_*(\Lambda)}(x) = \sum_{k \geq 1} a_k x^k$. 
Then  
\[
\prod\limits_{k=1}^\infty(1-x^k)^{-a_k}=\prod\limits_{s=1}^\infty\det C_\Lambda(x^s).
\]
\end{thm}
We achieve this by generalizing a result due to Igusa \cite{IGUSA1992101} that is a ``logarithmic'' version of the one in (A) due to Etingof and Ginzburg \cite{Etingof-Ginzburg'07}. 
More precisely, we check that Igusa's proof also works with these more general assumptions, and then we establish that this logarithmic version implies the one in the theorem above.

For our results relating to (B) and (C), we achieve our generalizations by replacing the notion of an almost Koszul algebra with a ``higher'' version which does not require the degree zero part of the algebra to be semisimple, only of finite global dimension: Almost $d$-$T$-Koszul algebras were introduced in \cite{HS} as a common generalization of the almost Koszul algebras of \cite{BBK02} and trivial extensions $\Delta(A)$ of so-called $d$-representation finite algebras $A$. 
Similarly to the case of a Koszul or almost Koszul algebra $\Lambda$, there is a notion of a dual algebra $\Lambda^!$. 
In the case of a trivial extension $\Delta(A)$ of a $d$-representation finite algebra $A$, the almost Koszul dual $\Delta(A)^!$ is the higher preprojective algebra $\Pi_{d+1}(A)$.

Roughly speaking, one can think of $d$-representation finite algebras as algebras of global dimension $d$ for which the Serre functor of their bounded derived categories behaves in a fashion that generalizes the behaviour of the Serre functor on the bounded derived category of representation finite hereditary algebras; see \cite{HIO14}. 
In both the classical case of $d = 1$ and for $d \geq 2$, the higher preprojective algebra of a $d$-representation finite algebra $A$ can be defined as the free tensor algebra over $A$ of a bimodule closely related to the Serre functor of $A$, namely $\Ext^{d}_A(D(A), A)$; see \cite{CB99} and \cite{IO13}, respectively.
Note that both of these are central classes of algebras that have been much studied; see e.g.\ \cite{Iya11, Iyama-Oppermann, HI11b, HI11, AO14, Grant-Iyama'20, DJW19, DJL21}.
Additionally, examples of both include classes of algebras that have been of interest to parts of theoretical physics since the early nineties; see e.g.\ \cite{Di-Francesco-Zuber, Ocneanu} and \cite{Caorsi-Cecotti}. 

Moreover, there are open conjectures and questions concerning these and related classes of algebras that can, in principle, be answered by studying their Hochschild cohomology. 
Indeed, recent work by Chan, Iyama and Marczinzik has shown that twisted fractionally Calabi--Yau algebras of finite global dimension can be characterized in terms of $d$-representation finite algebras and their higher preprojective algebras.
Moreover, they show that a conjecture stating that all $d$-representation finite algebras have acyclic quivers is equivalent to the same being true for twisted fractionally Calabi--Yau algebras.  
By \cite{Hap89}, it is known that, up to some assumptions, the first Hochschild cohomology of an algebra can detect cycles in its quiver.  

In \cite{Chan-et-al}, the same group showed that an algebra of finite global dimension is twisted fractionally Calabi--Yau if and only if its trivial extension is a twisted periodic algebra, thus connecting the twisted fractionally Calabi--Yau conjecture from \cite{HI11b} to the periodicity conjecture of Erdmann and Skowro\'{n}ski \cite{Erdmann-Skowronski'08}.
Both of the latter conjectures claim that certain algebra automorphisms have finite order as elements of the relevant outer automorphism group, and this too is something the first Hochschild cohomology can say something about by virtue of it being the space of outer derivations of the algebra.

We note that the algebras studied in \cite{Etingof-Eu'06,  Evans-Pugh'12} and \cite{Morigi'22} are related to higher preprojective algebras of $d$-representation finite algebras in the sense of Iyama and collaborators; see e.g.\ \cite{Iya11, Iyama-Oppermann, IO13}. 
Indeed, in the case of \cite{Etingof-Eu'06}, this is true by the fact that $d = 1$ recovers the case of representation finite hereditary algebras, i.e.\ those given by ADE Dynkin quivers.
The type $\mathbb{A}$ algebras considered in \cite{Evans-Pugh'12} and \cite{Morigi'22} are exactly the $(d+1)$-preprojective algebras of the higher type $\mathbb{A}$ algebras introduced by \cite{Iyama-Oppermann} in the case $d = 2$; moreover, the other types are evidently related to higher preprojective algebras even if they are not always themselves such algebras; see \cite{Ocneanu} and \cite{Haden'24}.

This suggests that it is natural to generalize this method to general higher preprojective algebras of $d$-representation finite algebras. 
Moreover, we note that there are other immediate benefits:
By the work of \cite{Dugas'12}, it is known that the higher preprojective algebra of a $d$-representation finite algebra always satisfies (D1) for $m = d + 1$. 
As for (D2), this is not known to always hold for higher preprojective algebras and is thus something that must be established case by case. 
Do note, however, that if an algebra satisfies (D1), property (D2) becomes equivalent to the aforementioned periodicity conjecture of Erdmann and Skowro\'{n}ski \cite{Erdmann-Skowronski'08} holding for that algebra.

Assume now and for the remainder of this introduction that the base field $K$ is algebraically closed. The following is our first step towards a replacement for (B).

\begin{prop}[See \cref{prop: graded version of Happel's result}]\label{iprop: graded version of Happel's result}
Let $A$ be an $\ell$-homogeneous $d$-representation finite algebra and write the integer $i$ as $i = (d+2)q + r$ for integers $q, r$ such that $0 \leq r < d + 2$. 
Then the graded minimal projective bimodule resolution $P^
{\bullet}(\Pi)$ of $\Pi :=  \Pi_{d+1}(A)$ satisfies that $P^{-i}(\Pi)$ is generated in degree $q\ell$ if $r = 0$ and in degrees $q\ell$ and $q\ell + 1$ otherwise.
\end{prop}

Note that the $\ell$-homogeneous assumption here is a condition on the orbits of the indecomposable projectives of $A$ with respect to the action of the Serre functor of $A$. 
In particular, if $d = 1$ and $A$ is of type $\mathbb{A}_n$, then $A$ is $\ell$-homogeneous if $n$ is odd and the quiver of $A$ has e.g.\ a bipartite orientiation. 
The preceding proposition is proved by using that this assumption in combination with Theorem 1.2 of \cite{AO14} allows us to deduce that $\Omega_{\Pi^{e}}^{d + 2}\Pi \cong \Pi_{\mu^{-1}}\langle \ell \rangle$, following which we apply some of the ideas in section 3.1 and 3.2 of \cite{AIR15}.

The bounds on the degrees of the generators can be improved when considering Hochschild homology. 
Hence, it is the following that will effectively serve as a version of the ``piece-wise linearity'' in (B). 

\begin{prop}[See \cref{prop: HH_* bounds on degrees}]\label{iprop: HH_* bounds on degrees}
Let $A$ be an $\ell$-homogeneous $d$-representation finite algebra, let $i$ be a non-negative integer, and write the integer $i$ as $i = (d+2)q + r$ for integers $q, r$ such that $0 \leq r < d + 2$. 
Then $$\HH_{i}(\Pi_{d+1}(A)) = \bigoplus_{j = q\ell}^{q\ell + \ell - 1}\HH_{i, j}(\Pi_{d+1}(A))$$ if $r = 0$  and $$\HH_{i}(\Pi_{d+1}(A)) = \bigoplus_{j = q\ell + 1}^{q\ell +\ell}\HH_{i, j}(\Pi_{d+1}(A))$$ otherwise. 
\end{prop}

We now present part of our result relating to (C), albeit in a simplified form which also requires the $\ell$-homogeneous assumption. 
For this, note that if $A$ is a $K$-algebra of finite global dimension, then $\Phi_A$ is here the Coxeter transformation of $A$, and recall that it is the linear transformation induced on  $\K_0(\D^b(\modu A))$ by the action of the Auslander--Reiten translate of $A$, i.e.\ by the action of the Serre functor of $A$ desuspended by $[-1]$.  

\begin{prop}[See \cref{prop: product of graded cartan matrices formula}]\label{iprop: product of graded cartan matrices formula}
Let $A$ be an $\ell$-homogeneous $d$-representation finite algebra. 

Then 
$$\det(C_{\Pi_{d+1}(A)}(x)C_{\Delta(A)}((-1)^{d+1}x)) = \det(I - ((-1)^{d-1} x\Phi^{-1}_{A})^{\ell}).$$ 
\end{prop}

Keeping in mind that $\Delta(A)$ is almost $(d+1)$-$A$-Koszul and $\Pi_{d+1}(A)$ is its dual algebra, it is fairly evident how this is a version of (C) above. 

Let $\phi_A(x)$ denote the Coxeter polynomial of $A$. 
The following result is one reason for why the preceding proposition is often quite easy to use in practice.  
\begin{prop}[See \cref{prop: det of graded Cartan matrix of trivial extension}]\label{iprop: det of graded Cartan matrix of trivial extension}
Let $A$ be a basic $K$-algebra of finite global dimension and let $\Delta(A)$ be the trivial extension of $A$. 
Then $$\det C_{\Delta(A)}(x) =   \phi_A(x)\det C_{A}.$$ 
\end{prop}

We now illustrate how this all comes together by way of a concrete example. 
Note that in the following, at no point do we explicitly compute a resolution, apply a functor and compute (co)homology. 
In fact, we get by with explicitly computing only the zeroth Hochschild cohomology, i.e.\ the center of the algebra. 

\begin{example}
Let $Q$ be an orientation of a Dynkin diagram of type $\mathbb{D}_4$. 
Then $KQ$ is a $1$-representation finite algebra that is $\ell$-homogeneous for $\ell = 3$ by Proposition 3.2 of \cite{HI11b}.
Consequently, $A = KQ \otimes_K KQ$ is then a $3$-homogeneous $2$-representation finite  algebra by Corollary 1.5 of \cite{HI11b}, and we now want to compute the Hochschild homology of its $(2 + 1)$-preprojective algebra 
$$\Lambda := \Pi_{2 + 1}(A) \cong A \oplus \Ext^{2}_A(DA, A) \oplus \Ext^{2}_A(DA, A)^{\otimes_A 2} \oplus \cdots.$$
By using a K\"{u}nneth formula and that $K$ is a field, $\Lambda$ can be shown to be isomorphic as a graded algebra to the Segre product of the preprojective algebra $$\Pi(KQ) \cong KQ \oplus \Ext^1_{KQ}(D(KQ), KQ) \oplus \Ext^1_{KQ}(D(KQ), KQ)^{\otimes_{KQ} 2} \oplus \cdots$$ 
with itself, i.e.\ we have $\Lambda \cong \oplus_{i\geq 0} \Pi(KQ)_i \otimes_K \Pi(KQ)_i$; see also \cite{Thi20} or \cite{HIO14} for related results. 
Note that we are here using the (higher) preprojective gradings, and that these are different from the ones used in the example before. 
Indeed, e.g.\ here the degree $0$ part of $\Pi(KQ)$ is all of $KQ$ and not just $KQ_0$, the latter denoting the semisimple part of $KQ$. 

Using this isomorphism and the description of Nakayama automorphism of $\Pi(KQ)$ as in e.g.\ \cite{Eu'07P}, it is straightforward to see that the Nakayama automorphism of $\Lambda$ can be chosen to be the identity. 
In other words, $\Lambda$ is a symmetric algebra, and in fact, it is a graded symmetric algebra, meaning that, since $\ell - 1$ is the highest degree of $\Lambda$, $D(\Lambda)$ and $\Lambda \langle {-\ell+1} \rangle$ are isomorphic as graded bimodules.  
Hence, we obtain that $\Omega^{4}_{\Lambda^{e}}\Lambda \cong \Lambda \langle 3 \rangle$ holds by using our \cref{iprop: graded version of Happel's result}. 
Moreover, in addition to formulas similar to the ones we had in the previous example, we also have $\HH^{i}(\Lambda) \cong D(\HH_{i}(\Lambda))\langle 2 \rangle$ for all $i$. 

Using results from \cite{Happel'97} that we recall in \cref{sec: some linear algebra}, we can show that $$\phi_A(-x) = (1 - x) (1 - x^{3})^{5}.$$ 
Moreover, we obtain that $\Phi_A^3 = -I$ by e.g.\ Proposition 2.4 of \cite{HI11b}. 
By combining our \cref{iprop: product of graded cartan matrices formula} and \cref{iprop: det of graded Cartan matrix of trivial extension}, we see that 
\begin{align*}
\det C_\Lambda (x)
& = \phi_A(-x)^{-1} \det(I - x^3I) \\
& = (1 - x)^{-1}(1 - x^3)^{-5} (1 - x^3)^{16}\\
& = (1 - x)^{-1}(1 - x^3)^{11}.
\end{align*}
Consequently, using our \cref{ithm: etingof-ginzburg} yields that 
\begin{align*}
\chi_{\overline{\HC}_*(\Lambda)}(x) 
& = \frac{x}{1 - x} - \frac{11x^3}{1 - x^3}\\
& = \frac{1}{1 - x^3}(x + x^2 + x^3 - 11x^3)\\
& = \frac{1}{1 - x^3}(x + x^2 - 10x^3).
\end{align*}

Moreover, we obtain the following diagram from the Connes long exact sequence. 

\[
\begin{tikzcd}
& & &\\
 & 0 \dar & \\
0 \leq \deg \leq \ell - 1 & \overline{\HH}_0(\Lambda) \dar["B_0"] \rar[equal, shorten = 2mm, shift right] & C \dar["\sim" {rotate=90, anchor=north}] &\\
1 \leq \deg \leq \ell & \overline{\HH}_1(\Lambda) \dar["B_1"] \rar[equal, shorten = 2mm, shift right] & C \rar[phantom, "\oplus"] & X \dar["\sim" {rotate=90, anchor=north}]\\
1 \leq \deg \leq \ell & \overline{\HH}_2(\Lambda) \dar["B_2"] \rar[equal, shorten = 2mm, shift right] & D(C)\langle \ell \rangle \dar["\sim" {rotate=90, anchor=north}] \rar[phantom, "\oplus"] & D(X)\langle \ell \rangle\\
1 \leq \deg \leq \ell & \overline{\HH}_3(\Lambda) \dar["B_3"] \rar[equal, shorten = 2mm, shift right] & D(C)\langle \ell \rangle \rar[phantom, "\oplus"] & L \langle \ell \rangle \dar["\sim" {rotate=90, anchor=north}]\\
\ell \leq \deg \leq 2\ell - 1 & \overline{\HH}_4(\Lambda)
\dar["B_4"] \rar[equal, shorten = 2mm, shift right]  & C \langle \ell \rangle \dar["\sim" {rotate=90, anchor=north}] \rar[phantom, "\oplus"] & L \langle \ell \rangle\\
 \ell + 1 \leq \deg \leq 2\ell & \overline{\HH}_5(\Lambda)  \rar[equal, shorten = 2mm, shift right] & C\langle \ell \rangle \rar[phantom, "\oplus"] & X \langle \ell \rangle
\end{tikzcd}
\]

We obtain the degree bounds in the left column by our \cref{iprop: HH_* bounds on degrees}. 
From these bounds, we observe that $L\langle \ell \rangle$ is trivial except in degree $\ell$, and that both $C$ and $X \cong D(X)\langle \ell \rangle$ are trivial in degrees less than $1$ and greater than $\ell - 1$. 
Similarly to in the first example, we thus deduce that $C, D(C)\langle \ell \rangle, L \langle \ell \rangle$ and $X \cong D(X)\langle \ell \rangle$ are only non-trivial in disjoint sets of degrees. 

As a consequence, we deduce that $h_L(x) = 10$ from our computation of the Euler characteristic above. 
Moreover, by using the description of $\Lambda$ as a Segre product, one can compute that $h_{\HH^0(\Lambda)}(x) = 1 + 4^2 x^{2}$, thus implying that $h_{\HH_0(\Lambda)}(x) =  4^2 +  x^{2}$ so that $h_C(x) = x^2$. 
Hence, $h_{D(C)\langle \ell\rangle} (x) = x$. 
Consequently, we see that $h_X(x) = 0$ and so $X$ must be trivial. 
With this we know $\HH_i(\Lambda)$ for $i = 0, \ldots, 4$. 
Moreover, by using that $\Omega^4_{\Lambda^{e}} \Lambda \cong \Lambda\langle 3 \rangle$, we have that $\HH_{i + 4}(\Lambda) \cong \HH_i (\Lambda) \langle \ell \rangle$ holds for $i \geq 1$. 
We also know the cyclic homology and can obtain the Hochschild cohomology by using the formula $\HH^{i}(\Lambda) \cong D(\HH_{i}(\Lambda))\langle -2 \rangle$ for all $i$ that holds as a consequence of $\Lambda$ being graded symmetric of highest degree $2$.
\end{example}

We note that while this example could be handled using existing methods, we believe that this would ultimately be more cumbersome than what is done here: Indeed, to use existing methods, one could use the fact that $A$ here is Koszul in combination with one of the main results of \cite{Grant-Iyama'20} to easily deduce that $\Lambda$ is an almost Koszul algebra in the sense of \cite{BBK02}. 
However, one would then need to compute the determinant of the graded Cartan matrix of the almost Koszul dual $\Lambda^!$. 
While in this case $\Lambda^! \cong \Delta(A^!)$ where $A$ is considered as a classical Koszul algebra as in \cite{BGS96}, this nevertheless seems trickier than what we do above, and it is perhaps much more difficult in general:
Indeed, one of the conjectural results Morigi makes use of in \cite{Morigi'22} is a formula for the determinant of the graded Cartan matrix of $\Lambda^!$; see Conjecture 5.3.0.2 of \cite{Morigi'22}. 

In fact, we can prove that ``our'' version of this conjecture holds when the parameter $s$ in Conjecture 5.3.0.2 of \cite{Morigi'22} is even: Namely, by considering $\Lambda = \Pi_{2 + 1}(A)$ as a higher almost Koszul algebra with respect to $A$, we can use \cref{iprop: det of graded Cartan matrix of trivial extension} in combination with Corollary 1.13 from \cite{Ladkani'12}.
Moreover, we can thus prove that the ``even case'' of the results in \cite{Morigi'22} hold without assuming the conjectures that \cite{Morigi'22} does. 
Using the general form of \cref{iprop: product of graded cartan matrices formula}, we believe we can also show the ``odd case'' of the results in \cite{Morigi'22}, although this is slightly more involved. 
This will be done in a followup paper. 

We now come to the final result we highlight in this introduction. 
Namely, in the last section of the paper, we compute the cyclic homology, and the Hochschild homology and cohomology for the higher preprojectives of the $2$-hereditary algebras obtained as tensor products of representation finite hereditary algebras of type $\mathbb{A}_n$.
Recall that by \cite{HI11b}, we know that each tensor factor must have the same odd parameter $n = 2\ell - 1$, and the quiver of each factor must be oriented symmetrically as in Proposition 3.2 of \cite{HI11b}.

\begin{thm}[See \cref{thm: cychlic and hochschild homology of higher preproj of type A x A} and \cref{thm: hochschild cohomology of higher preproj of type A x A}] \label{ithm: (co)homology of higher preproj of tensor products of type A}
Let $\Pi = \Pi_{2 + 1}(K\mathbb{A}_{2\ell - 1}^{\otimes 2})$ with $\ell \geq 2$ and $\mathbb{A}_{2\ell - 1}$ oriented symmetrically, and let $\Pi$ be endowed with the higher preprojective grading.
In addition, assume that the base field $K$ is of characteristic zero.
Then the cyclic homology of $\Pi$ is given by 
\[
\begin{array}{lcl}
\HC_0(\Pi) \cong \Pi/\rad \Pi,
 & \qquad &
\HC_1(\Pi) \cong 0,\\
\HC_2(\Pi) \cong X_2,
& & 
\HC_3(\Pi) \cong 0,\\
\HC_4(\Pi) \cong D(X_2)\langle 2\ell\rangle,
& & 
\HC_5(\Pi)  \cong 0,\\
\HC_6(\Pi)  \cong 0,
& &
\HC_7(\Pi) \cong L\langle 2\ell \rangle,\\
\HC_8(\Pi) \cong 0,
& &
\HC_{i + 8}(\Pi) \cong \HC_{i}(\Pi)\langle 2\ell \rangle, \quad i \geq 1,
\end{array}
\]
where $h_{\Pi/\rad \Pi}(x) = (2\ell - 1)^2$, $h_{X_2}(x) = \sum_{i = 1}^{\ell}x^{i}$ and $h_{L}(x) = \frac{(2\ell - 2)^2}{2}$. 

Moreover, the Hochschild homology of $\Pi$ is given by 
\[
\begin{array}{lcl}
\HH_0(\Pi)  \cong \Pi/\rad \Pi, & \qquad &
\HH_1(\Pi)  \cong 0, \\
\HH_2(\Pi)  \cong X_2, & & \HH_3(\Pi)  \cong X_2,\\
\HH_4(\Pi)  \cong D(X_2)\langle 2\ell\rangle, & &
\HH_5(\Pi)  \cong D(X_2)\langle 2\ell\rangle,\\
\HH_6(\Pi)  \cong 0, & &
\HH_7(\Pi)  \cong L\langle 2\ell \rangle, \\
\HH_{8}(\Pi)  \cong L\langle 2\ell \rangle, & & \HH_{i + 8}(\Pi) \cong \HH_{i}(\Pi)\langle 2\ell \rangle, \quad i \geq 1,\\
\end{array}
\]
where $\Pi/\rad \Pi, X_2$ and $L$ are the same as before. 

Finally, the Hochschild cohomology of $\Pi$ is given by
\[
\begin{array}{lcl}
\HH^0(\Pi)  \cong Z, 
& \qquad &
\HH^1(\Pi)  \cong X_2\langle - 1\rangle, \\
\HH^2(\Pi)  \cong 0, 
& & 
\HH^3(\Pi)  \cong L\langle - 1\rangle,\\
\HH^4(\Pi)  \cong L\langle - 1\rangle, 
& &
\HH^5(\Pi)  \cong 0,\\
\HH^6(\Pi)  \cong D(X_2)\langle -1 \rangle, 
& &
\HH^7(\Pi)  \cong D(X_2)\langle -1 \rangle, \\
\HH^{8}(\Pi)  \cong  X_2\langle -2\ell - 1\rangle,
& & 
\HH^{i + 8}(\Pi) \cong \HH^{i}(\Pi)\langle 2\ell \rangle, \quad i \geq 1,\\
\end{array}
\]
where $h_{Z}(x) = \sum_{i = 0}^{\ell- 1}x^{i}$ and $X_2$ and $L$ are the same as before.
\end{thm}
To do this, we recall a description of the higher preprojective algebra of a tensor product of two $d$-representation finite algebras as a Segre product. 
Using this description, we compute $\HH_0$ and $\HH^0$ of such a higher preprojective algebra $\Pi$. 
The result then follows by computing the graded Euler characteristic of $\Pi$ and applying our version of the method. 
In particular, as before it is the case that at no point do we explicitly compute a resolution, apply a functor and compute (co)homology. 

In a second followup paper, we will compute the Hochschild (co)homology and the cyclic homology of higher preprojective algebras of the remaining $2$-representation finite algebras that arise as tensor products of ADE Dynkin quiver algebras. 
The example above is, of course, one particular instance of such an algebra.

This paper is organized as follows. 
In \cref{sec: preliminaries}, we review graded algebras, Hochschild (co)homology, and cyclic homology. 
\cref{sec:product formula for the euler characteristic} consists of showing \cref{ithm: etingof-ginzburg} by showing --- mutatis mutandis --- that Igusa's proof of the logarithmic version of \cref{ithm: etingof-ginzburg} also works in our greater generality. 
To provide a foundation for the remainder of the paper, \cref{sec: On Frobenius algebras and d-hereditary algebras} reviews Frobenius algebras, $d$-hereditary algebras, and stable Hochschild (co)homology. 
Moreover, we end this section by showing \cref{iprop: graded version of Happel's result} and \cref{iprop: HH_* bounds on degrees}.
In \cref{sec: some linear algebra} we begin by recalling some facts about Coxeter polynomials, following which we show \cref{iprop: det of graded Cartan matrix of trivial extension} and \cref{iprop: product of graded cartan matrices formula}.
Finally, in \cref{sec:computing homology}, we begin by recalling a description of higher preprojective algebras of $d$-hereditary algebras obtained via tensor products as Segre products.
We then prove \cref{ithm: (co)homology of higher preproj of tensor products of type A} by computing the zeroth Hochschild homology, the zeroth Hochschild cohomology, and the graded Euler characteristic of the algebras involved and then applying our version of the method. 
\end{example}

\subsection{Conventions and notation}
We let $\Lambda$ denote a $K$-algebra over a field $K$. 
In the sequel, the field $K$ will often be assumed to be of characteristic zero, and in, certain (sub)section(s), perfect or even algebraically closed. Although it should usually be of little consequence in this paper, we note that we work with right modules unless something else is stated.    

\section{Preliminaries}\label{sec: preliminaries}
In this section, we begin by setting notation and recalling some facts about graded algebras. 
Following this, we review Hochschild homology and cohomology and recall how a grading on $\Lambda$ induces gradings on its Hochschild homology and cohomology. 
Then we recall a construction of cyclic homology.
Note, however, that we simplify things by assuming that the base field has characteristic zero since we will need to assume that in our applications anyway. 
Finally, we recall how one can introduce reduced forms of Hochschild and cyclic homology in case $\Lambda$ is a graded algebra. 

\subsection{Graded vector spaces}
Let $V = \oplus_{i \ge 0} V_i$ be a positively graded --- i.e.\ $\mathbb{N}$-graded --- $K$-vector space and recall that the \textit{Hilbert series of} $V$ is given by the formal power series 
\[
h_V(x) := \sum_{i \ge 0} \dim_K (V_i) x^{i} \in \mathbb{Z}[[x]],
\]
assuming that each vector space $V_i$ is finite dimensional.

Let $V^* =\{ V^{(j)} \}_{j \in \mathbb{Z}}$ be a $\mathbb{Z}$-indexed family of graded vector spaces $V^{(j)}$. Furthermore, assume that each homogeneous space $V_i^{(j)}$ is finite dimensional, and that for each homogeneous degree $i$, there is only a finite number of indices $j$ for which $V_i^{j} \ne 0$. Then the \emph{graded Euler characteristic} of $V^*$ is
$$
\chi_{V^*}(x) = \sum_{j \in \mathbb{Z}} (-1)^j h_{V^{(j)}}(x).
$$

\subsection{Graded algebras and modules}

Assume now that $\Lambda = \oplus_{i \ge 0}\Lambda_i$ is a positively graded algebra. 
We let $\Gr \Lambda$ denote the category consisting of graded $\Lambda$-modules and homogeneous homomorphisms of degree $0$, and we let $\gr \Lambda$ denote its subcategory consisting of finitely presented modules. 
Recall that the latter is abelian if and only if $\Lambda$ is graded right coherent. 

Recall that given a graded $\Lambda$-module $M$, we let $M \langle j \rangle$ denote the $j$\textit{-th graded shift of} $M$, i.e.\ $M \langle j \rangle$ has the same underlying $\Lambda$-module structure as $M$, but its grading is defined by setting $M\langle j \rangle_i := M_{i - j}$.
The following result allows us to describe ungraded extensions in terms of graded ones. 

\begin{lem}[{\cite[Corollary 2.4.7]{Nastasescu-Van-Oystaeyen}}] \label{lem: ext in terms of graded ext}
Let $M$ and $N$ be graded $\Lambda$-modules. If $M$ is finitely generated and there is a projective resolution of $M$ such that all syzygies are finitely generated, then
\[
\Ext_{\Lambda}^{i}(M,N) \cong \bigoplus_{j\in \mathbb{Z}}\Ext_{\Gr \Lambda}^{i}(M,N\langle -j \rangle)
\]
for all $i \geq 0$.
\end{lem}

Similarly, for graded modules $M$ and $N$ we have
$$
\Tor_i^\Lambda(M, N) \cong \bigoplus_{j \in \mathbb{Z}} \Tor_i^{\Gr(\Lambda)}(M, N\langle -j \rangle),
$$
where $\Tor_i^{\Gr \Lambda}$ is the derived functor of the degree $0$ part of the (graded) tensor product of two graded $\Lambda$-modules; see e.g.\ page 30 of \cite{Nastasescu-Van-Oystaeyen} or Section 22.12 of \cite{stacks-project} for the definition of the grading on the tensor product of graded modules. 
In other words, if $M \in \Gr \Lambda$, then $\Tor^{\Gr \Lambda}_0(\Lambda, M\langle -j \rangle) = M_j$.
Moreover, if $\Lambda$ is e.g.\ finite dimensional and $M, N \in \gr \Lambda$, then $D\Ext^{i}_{\Gr \Lambda}(M, D(N)\langle j\rangle) \cong \Tor_{i}^{\Gr \Lambda}(M, N\langle -j\rangle)$.

If $\Lambda = \oplus_{i = 0}^{a}\Lambda_i$ is a basic finite dimensional algebra with $e_1, e_2, \ldots, e_n \in \Lambda_0$ a complete list of primitive orthogonal idempotents for the degree zero component of the algebra, the \emph{graded Cartan matrix of} $\Lambda$ is given by the matrix polynomial
\[
C_{\Lambda}(x) = \sum_{k  = 0}^{a}C^{(k)} x^{k},
\]
where $C^{(k)}$ is the matrix given by
$$C^{(k)}_{i,j} := \dim_K(e_i \Lambda_k e_j).$$
We note that any finite dimensional algebra may be viewed as a graded algebra concentrated in degree zero, in which case the graded Cartan matrix defined above is equal to the integer matrix $C^{(0)}$, which we call the \emph{(ungraded) Cartan matrix} of the underlying algebra. 
Furthermore, when $\Lambda$ is an infinite dimensional graded algebra, the definition of the graded Cartan matrix, mutatis mutandis, also makes sense up to some assumptions on $\Lambda$, e.g.\ if we assume that $\Lambda$ is locally finite dimensional --- meaning that each of its homogeneous components is finite dimensional --- and that $\Lambda_0$ is basic. In this case, the graded Cartan matrix of $\Lambda$ has entries in $\mathbb{Z}[[x]]$. 

Note that when $\Lambda_0 = K$, the graded Cartan matrix of $\Lambda$ is simply the Hilbert series of $\Lambda$ as a positively graded vector space. 

\subsection{Hochschild homology and cohomology}
Since $\Lambda$ is an algebra over a field, one can define its Hochschild homology and cohomology via derived functors. 
Hence, recall that the \textit{enveloping algebra of} $\Lambda$ is given by $\Lambda^{e} := \Lambda^{\op} \otimes_K \Lambda$, and that $\Lambda$-bimodules can be understood as (right) modules over $\Lambda^{e}$. 
Given this, the $i$th Hochschild homology and cohomology of $\Lambda$ can then be given by, respectively,  
\[
\HH_{i}(\Lambda) := \Tor_{i}^{\Lambda^{e}}(\Lambda, \Lambda)
\]
and
\[
\HH^{i}(\Lambda) := \Ext^{i}_{\Lambda^{e}}(\Lambda, \Lambda).
\]

In case $\Lambda$ is a graded algebra, we consider the grading on $\Lambda^{e}$ obtained by setting $$\Lambda^{e}_i := \oplus_{j + k = i} \Lambda_j \otimes_K \Lambda_k,$$
i.e.\ a K{\"u}nneth grading.
Provided $\Lambda$ satisfies sufficient ``finiteness'' assumptions allowing one to apply e.g.\ \cref{lem: ext in terms of graded ext}, this then induces a grading on $\HH_{i}$ and $\HH^{i}$  by setting 
\[
\HH_{i,j}(\Lambda) := \Tor_{i}^{\Gr \Lambda^{e}}(\Lambda, \Lambda \langle {-j} \rangle)
\]
and
\[
\HH^{i,j}(\Lambda) := \Ext^{i}_{\Gr \Lambda^{e}}(\Lambda, \Lambda \langle -j \rangle).
\]
\subsection{Cyclic homology} \label{sec:cyclic homology}

We now recall the construction of cyclic homology, following \cite{loday}. We note that although cyclic homology can be defined for algebras over any field (or even over a commutative ring), the construction we present here applies only to the case where the field $K$ has characteristic zero.  

There is a complex of $K$-vector spaces, called the \textit{Hochschild complex}, of the form 
$$
\cdots \overset{b}{\to} \Lambda^{\otimes3} \overset{b}{\to} \Lambda^{\otimes 2} \overset{b}{\to} \Lambda,
$$
where $\Lambda$ is in homological degree zero. Here $b$ is the Hochschild boundary, as defined in \cite{loday}. We consider the cyclic operator $t$ on $\Lambda^{\otimes i + 1}$, whose action on elementary tensors is defined by
$$
t(a_0 \otimes \ldots \otimes a_i) = (-1)^i a_i \otimes a_0 \otimes \ldots \otimes a_{i - 1}.
$$
It can be shown that $b$ induces a well defined map $\Lambda^{\otimes i + 1}/(1 - t) \to \Lambda^{\otimes i}/(1 - t)$, so there is a complex
$$
\cdots \overset{b}{\to} \Lambda^{\otimes 3}/(1 - t) \overset{b}{\to} \Lambda^{\otimes 2} /(1 - t) \overset{b}{\to} \Lambda,
$$
known as the \textit{Connes complex}, which is denoted by $C^\lambda(\Lambda)$. 
The $i$th cyclic homology of $\Lambda$ is defined to be $\HC_i(\Lambda) = \Hm_i(C^\lambda(\Lambda))$.
If $\Lambda$ is a graded algebra, then the vector spaces and maps in the Connes complex are also graded, and thus $\HC_i(\Lambda)$ admits a natural grading. 

We also consider the notion of relative cyclic homology. For an ideal $I \subseteq \Lambda$, the quotient map $\Lambda \to \Lambda/I$ induces a chain map $p \colon C^\lambda(\Lambda) \to C^\lambda(\Lambda/I)$.\textit{ The relative cyclic homology of} $\Lambda$ \textit{and} $I$, denoted $\HC_*(\Lambda, I)$, is the homology of the kernel of $p$. Thus $\HC_*(\Lambda, I)$ fits into a long exact sequence:
$$
\cdots \to \HC_i(\Lambda, I) \to \HC_i(\Lambda) \to \HC_i(\Lambda/I) \to \HC_{i - 1}(\Lambda, I) \to \cdots
$$
Much like cyclic homology, relative cyclic homology also admits a natural grading if $\Lambda$ is graded and $I$ is a homogeneous ideal.

\subsection{Reduced Hochschild and cyclic homology} \label{sec:reduced homology}
In this subsection, we consider a notion of reduced Hochschild or cyclic homology for a positively graded algebra $\Lambda$. Note thus that the inclusion map $\Lambda_0 \hookrightarrow \Lambda$ is a split monomorphism of algebras, and hence it induces a (split) monomorphism $\HH_*(\Lambda_0) \to \HH_*(\Lambda)$ in Hochschild homology. 
Thus, for any $i$, we may view $\HH_i(\Lambda_0)$ as a subspace of $\HH_i(\Lambda)$, and we define the $i$\textit{th reduced Hochschild homology of} $\Lambda$ to be $\overline{\HH}_i(\Lambda) = \HH_i(\Lambda)/\HH_i(\Lambda_0)$. Similarly, the $i$\textit{th reduced cyclic homology of} $\Lambda$ is defined to be $\overline{\HC}_i(\Lambda) = \HC_i(\Lambda)/\HC_i(\Lambda_0)$. Note that this notion is not necessarily the same as what Loday refers to as ``reduced'' Hochschild or cyclic homology in \cite{loday}.

In \cref{sec:product formula for the euler characteristic}, it will be useful to know how the homogeneous components of $\overline{\HC}_*(\Lambda)$ are related to those of $\HC_*(\Lambda)$. 
We therefore note that the inclusion map $\Lambda_0 \hookrightarrow \Lambda$ induces an isomorphism $$\HC_*(\Lambda_0) \overset{\cong}{\to} \HC_*(\Lambda)_0$$ in homogeneous degree zero, as can be seen from the construction of cyclic homology in the preceding  subsection. We thus see that the $k$th homogeneous component of the reduced cyclic homology of $\Lambda$ is
\begin{equation} \label{eq:value of reduced cyclic homology}
\overline{\HC}_{*, k}(\Lambda) \cong \begin{cases}
    0 & \text{if $k = 0$} \\
    \HC_{*, k}(\Lambda) & \text{if $k > 0$},
\end{cases}
\end{equation}
where the case $k > 0$ follows from the fact that $\HC_*(\Lambda_0)$ is concentrated in homogeneous degree zero.

If the field $K$ has characteristic zero, then for any $i$ there is a short exact sequence
$$
0 \to \overline{\HC}_{i - 1}(\Lambda) \to \overline{\HH}_i(\Lambda) \to \overline{\HC}_i(\Lambda) \to 0,
$$
by \cite[Theorem 4.1.13]{loday}. Note, however, that the cited source denotes $\overline{\HH}_i(\Lambda)$ and $\overline{\HC}_i(\Lambda)$ by $\overset{\approx}{\HH}_i(\Lambda)$ and $\overset{\approx}{\HC}_i(\Lambda)$, respectively. Splicing these short exact sequences together, we obtain the Connes long exact sequence:
$$
0 \to \overline{\HH}_0(\Lambda) \to \overline{\HH}_1(\Lambda) \to \overline{\HH}_2(\Lambda) \to \overline{\HH}_3(\Lambda) \to \cdots
$$

\section{A formula for the graded Euler characteristic} \label{sec:product formula for the euler characteristic}
In \cref{sec:computing homology}, we need to know the value of $\chi_{\overline{\HC}_*(\Lambda)}$, the graded Euler characteristic of the reduced cyclic homology of an algebra $\Lambda$.
In order to deduce the coefficients of this power series, we use an approach similar to one employed by Etingof and Eu \cite{Etingof-Eu'06}: in the case where $\Lambda$ is the preprojective algebra of an ADE Dynkin quiver, they use the formula
\begin{equation} \label{eq:product formula for coefficients of euler characteristic}
\prod_{k = 1}^\infty (1 - x^k)^{-a_k} = \prod_{s = 1}^\infty \det C_\Lambda(x^s),
\end{equation}
where $a_k$ is the coefficient of $x^k$ in the power series $\chi_{\overline{\HC}_*(\Lambda)}(x)$. In order to show that this formula also holds for the algebras that we are interested in in this paper, we will use the following formula that is due to Igusa \cite{IGUSA1992101}.
\begin{equation} \label{eq:igusa's formula - initial statement}
\log \det C_\Lambda(x) = \sum_{k = 1}^\infty \chi_{\HC_{*}(\Lambda, R)}(x^k) \sum_{d | k} \frac{d \mu(d)}{k}
\end{equation}
Here, $\mu$ is the M\"obius function (see \cref{sec:dirichlet convolution}), and $R$ is a suitable ideal of $\Lambda$. Note that in Igusa's paper, $R$ is chosen to be the Jacobson radical of $\Lambda$. However, Igusa's proof of this formula makes some assumptions that are not satisfied for the algebras we investigate in this paper. 
In particular, it assumes that the radical of the algebra is equal to the vector space spanned by the homogeneous elements of strictly positive degree. In this section, we therefore show that Igusa's formula holds for a more general class of graded algebras which includes the algebras under consideration in this paper.

Our proof follows Igusa's original proof quite closely, and in cases where the arguments would be identical to the ones given by Igusa, we sometimes omit some details, instead focusing on the parts that need to be modified in order to work in our more general context. We also present proofs of some lemmas that are stated but not proven in Igusa's paper. We note that although Igusa's results are stated for $\mathbb{N}^m$-graded algebras, we work with algebras graded by the natural numbers --- i.e.\ positively graded algebras --- as this is sufficient for our applications. However, our proofs can be readily adapted to the $\mathbb{N}^m$-graded case.

Throughout this section, let $\Lambda = \bigoplus_{i \geq 0} \Lambda_i$ be a positively graded $K$-algebra satisfying the following assumptions:
\begin{enumerate}[(i)]
    \item $\Lambda$ is locally finite dimensional, i.e.\ each homogeneous component $\Lambda_i$ of $\Lambda$ is finite dimensional.
    \item The homogeneous component of $\Lambda$ of degree zero is given by a finite acyclic quiver with admissible relations.
\end{enumerate}
Moreover, we assume that $K$ has characteristic zero.
We let $R = \rad(\Lambda_0) \oplus \Lambda_{>0}$, where $\Lambda_{>0}$ is the subspace of $\Lambda$ spanned by all homogeneous elements of positive degree. Additionally, we let $D = \Lambda/R \cong \Lambda_0/\rad(\Lambda_0)$. Note that the quotient map $\Lambda \to D$ admits a right inverse graded algebra homomorphism, which may be constructed using the composition
$$
D \overset{\cong}{\to} \Lambda_0/\rad(\Lambda_0) \to \Lambda_0 \hookrightarrow \Lambda,
$$
where the middle map is a right inverse of the quotient map $\Lambda_0 \to \Lambda_0/\rad(\Lambda_0)$. The existence of such a map $D \to \Lambda$ allows us to view $R$ as a graded $D$-bimodule.

For $k \ge 1$, we denote by $(R \otimes_D)^k$ the tensor product
$$
(\underbrace{R \otimes_D R \otimes_D \cdots \otimes_D R}_k) \otimes_{D^e} D.
$$
Moreover, we let $t$ be the operator on $(R \otimes_D)^k$ whose action on elementary tensors is defined by
$$
t((a_1 \otimes a_2 \otimes \ldots \otimes a_k) \otimes 1) = (-1)^{k + 1} (a_k \otimes a_1 \otimes \ldots \otimes a_{k - 1}) \otimes 1.
$$
Then we have the following result.

\begin{prop}[{\cite[Corollary 1.2]{IGUSA1992101}}]
    \label{prop:igusa's complex for relative cyclic homology}

The relative cyclic homology $\HC_*(\Lambda, R)$ is the homology of a complex of the following form.
\begin{equation} \label{eq: igusa's complex for relative cyclic homology}
R \otimes_{D^e} D \longleftarrow (R \otimes_D)^2 /(1 - t) \longleftarrow (R \otimes_D )^3/(1 - t) \longleftarrow \cdots
\end{equation}
    
\end{prop}

\subsection{Dimensional computations}
In this subsection, we prove several results concerning the dimensions of  the vector spaces appearing in the complex \eqref{eq: igusa's complex for relative cyclic homology}, with the goal of ultimately using this complex to compute the graded Euler characteristic of $\HC_*(\Lambda, R)$. In working towards this goal, it will be useful to have a better understanding of the structure of the space $(R \otimes_D)^k$.
We therefore note that
$$D \cong  \bigoplus_{i = 1}^n e_iK ,$$
where $e_1, \ldots, e_n$ are idempotents corresponding to the vertices of the quiver of $\Lambda_0$. 
For $D$-modules $M_D$ and $_D N$, we have
$$
M \otimes_D N \cong \bigoplus_{i = 1}^n Me_i \otimes_K e_i N.
$$
Moreover, for any $D$-bimodule $L$, we have
$$
L \otimes_{D^e} D \cong \bigoplus_{i = 1}^n e_i L e_i.
$$
In particular, it follows that
\begin{equation}\label{eq:iterated tensor product}
(R \otimes_D)^k \cong  \bigoplus_{i_1, \ldots, i_k} e_{i_1} R e_{i_2} \otimes_K \cdots \otimes_K e_{i_k} R e_{i_1}.
\end{equation}

As a consequence of \cref{prop:igusa's complex for relative cyclic homology}
and \eqref{eq:iterated tensor product}, we can already determine the degree zero component of $\HC_*(\Lambda, R)$:

\begin{lem}\label{lem:relative cyclic homology vanishes in homogeneous degree zero}
The relative cyclic homology $\HC_*(\Lambda, R)$ vanishes in homogeneous degree zero.
\end{lem}
\begin{proof}
By \cref{prop:igusa's complex for relative cyclic homology}, it suffices to show that $(R\otimes_D)^k$ vanishes in degree zero for all $k \ge 1$. By \eqref{eq:iterated tensor product}, the degree zero component of $(R \otimes_D)^k$ is
$$
\bigoplus_{i_1, \ldots, i_k} e_{i_1} R_0 e_{i_2}  \otimes \ldots \otimes e_{i_k}R_0 e_{i_1}.
$$
But $R_0$ is equal to $\rad(\Lambda_0)$, which is the ``arrow ideal'' of $\Lambda_0$, so this direct sum vanishes since the quiver of $\Lambda_0$ is acyclic.
\end{proof}

We let $E(x) = C_\Lambda(x) - I_n$, where $C_\Lambda(x)$ is the graded Cartan matrix of $\Lambda$ and $I_n$ is the $n \times n$ identity matrix. 
Then the following result allows us to determine the dimension of the homogeneous components of $(R \otimes_D)^k$.

\begin{lem}[{\cite[Lemma 3.1]{IGUSA1992101}}] \label{lem:dimension of homogeneous component of tensor product}
    The dimension of the degree $l$ component of $(R \otimes_D)^k$ is equal to the coefficient of $x^l$ in $\tr(E(x)^k)$.
\end{lem}
\begin{proof}
We claim that for any indices $i_1$ and $i_{k + 1}$, the dimension of
\begin{equation} \label{eq:direct sum of tensor products}
\bigoplus_{i_2, \ldots, i_k} (e_{i_1} R e_{i_2} \otimes_K \ldots \otimes_K e_{i_k} R e_{i_{k + 1}})_l
\end{equation}
is equal to the coefficient of $x^l$ in the $(i_1, i_{k + 1})$ entry of $E(x)^k$. Given this claim, the result follows from \eqref{eq:iterated tensor product} by setting $i_{k + 1}$ equal to $i_1$ and taking the sum of the dimension of \eqref{eq:direct sum of tensor products} over all possible values of $i_1$.

We prove the claim by induction on $k$. Since the quotient map $\Lambda \to D$ splits as a map of graded algebras, the dimension of the degree $l$ component of $e_i R e_j$ must be equal to the coefficient of $x^l$ in the $(i, j)$ entry of $C_\Lambda(x) - C_D(x)$, where $C_D(x)$ is the graded Cartan matrix of $D$. But $C_D(x)$ is the identity matrix, so $C_\Lambda(x) - C_D(x)$ is equal to $E(x)$. This proves the claim in the case $k = 1$.

Now suppose that the claim holds for $k -1$, where $k > 1$. We compute the dimension of \eqref{eq:direct sum of tensor products} as
$$
\sum_{m \in \mathbb{N}} \sum_{i_2} \dim(e_{i_1} R e_{i_2})_m \cdot \dim \left(\bigoplus_{i_3, \ldots, i_k} e_{i_2} R e_{i_3} \otimes_K  \cdots \otimes_K e_{i_k} R e_{i_{k + 1}} \right)_{l - m }.
$$
For a power series $p(x)$ over a ring $A$, we let $\Coeff_{x^l}(p(x)) \in A$ denote the coefficient of $x^l$ in $p(x)$. Then by the inductive assumption, the sum above can be rewritten as
\begin{align*}
& \sum_{m} \sum_{i_2} \Coeff_{x^m}(E(x)_{i_1, i_2}) \Coeff_{x^{l - m}} ((E(x)^{k - 1})_{i_2, i_{k + 1}}) \\
= & \sum_{m} \sum_{i_2} \Coeff_{x^m}(E(x))_{i_1, i_2} \Coeff_{x^{l - m}} ((E(x)^{k - 1}))_{i_2, i_{k + 1}},
\end{align*}
where on the last line we are viewing $E(x)$ and $E(x)^{k - 1}$ as power series over the matrix ring $M_n(\mathbb Z)$, 
so that $\Coeff_{x^m}(E(x))$ and $\Coeff_{x^{l - m}}(E(x)^{k - 1})$ denote matrices over $\mathbb Z$. The sum indexed by $i_2$ is just the definition of the $(i_1, i_{k +1})$ entry of the product of $\Coeff_{x^m}(E(x))$ and $\Coeff_{x^{l - m}}(E(x)^{k - 1})$, so we get
\begin{align*}
    & \sum_m (\Coeff_{x^m} (E(x)) \Coeff_{x^{l-m}} (E(x)^{k - 1}))_{i_1, i_{k + 1}} \\
    =& \left ( \sum_m \Coeff_{x^m} (E(x)) \Coeff_{x^{l-m}} (E(x)^{k - 1}) \right )_{i_1, i_{k + 1}} \\
    =& \Coeff_{x^l}(E(x) E(x)^{k - 1})_{i_1, i_{k + 1}} \\
    =& \Coeff_{x^l} (E(x)^k)_{i_1, i_{k + 1}} \\
    =& \Coeff_{x^l} ((E(x)^k)_{i_1, i_{k + 1}}),
\end{align*}
as desired.
\end{proof}

In order to compute the dimension of the homogeneous components of $(R \otimes_D)^k/(1 - t)$, we construct an appropriate basis for $(R \otimes_D)^k$ as follows. We first choose a homogeneous basis for $e_i R e_j$ for each $i$ and $j$.  Then by \eqref{eq:iterated tensor product}, there is a homogeneous basis for $(R \otimes_D)^k$  given by elements of the form $(b_1 \otimes b_2 \otimes \ldots \otimes b_k) \otimes 1$, where $b_i$ is a basis element of $e_{j_i} R e_{j_{i + 1}}$, for some indices $j_1, j_2, \ldots, j_{k +1}$ where $j_{k + 1} = j_1$. This basis is closed up to sign under the action of $t$. Moreover, each basis element is fixed up to sign by $t^p$ for some $p > 0$, and we call the smallest such $p$ the \textit{period of the basis element}. We observe that the period of a basis element is necessarily a divisor of $k$.

\begin{lem}[{\cite[Lemma 3.3]{IGUSA1992101}}]
\label{lem:counting basis elements of period p}

For a divisor $p$ of $k$, the number of basis elements of $(R \otimes_D)^k$ of period $p$ and of homogeneous degree $l$ is equal to the coefficient of $x^l$ in
$$
\sum_{d | p} \mu(d) \tr(E(x^{dk/p})^{p/d}).
$$

\end{lem}
\begin{proof}
Let $F(p)$ denote the number of basis elements of degree $l$ in $(R \otimes_D)^k$ whose period is exactly $p$, and let $G(p)$ be the number of degree $l$ basis elements whose period is a divisor of $p$. Then $G(p)$ is equal to the number of basis elements in $(R \otimes_D)^{p}$ of degree $\frac{p l}{k}$, i.e.\ the coefficient of $x^{pl/k}$ in $\tr(E(x)^{p})$ by \cref{lem:dimension of homogeneous component of tensor product}; equivalently, $G(p)$ is the coefficient of $x^l$ in $\tr(E(x^{k/p})^p)$. Since $G(p) = \sum_{d|p} F(d),$ the claim now follows from the M\"{o}bius inversion formula (see \cref{sec:dirichlet convolution}).
\end{proof}

By using \cref{lem:counting basis elements of period p}, we can now prove the following result, which allows us to determine the dimension of the homogeneous components of $(R \otimes_D)^k / (1 - t)$. 

\begin{lem}[{\cite[Lemma 3.4]{IGUSA1992101}}]
    \label{lem:dimension of quotient space in complex}
The dimension of the homogeneous component of $(R \otimes_D)^k/(1 - t)$ of degree $l$ is equal to the coefficient of $x^l$ in
$$
\sum_{p |  k} \frac{1}{p} \sum_{d | p} \mu(d) \tr(E(x^{dk/p})^{p/d})
$$
if $k$ is odd, or in
$$
\sum_{\substack{p | k \\ p \text{ even}}} \frac{1}{p} \sum_{d | p} \mu(d) \tr(E(x^{dk/p})^{p/d})
$$
if $k$ is even.
\end{lem}
\begin{proof}
For a basis element $b$ in $(R \otimes_D)^k$, we let the \emph{orbit} of $b$ be the set of basis elements that are equal to $\pm t^j (b)$ for some $j$. Let $\{b_1, \ldots, b_p\}$ be an orbit of basis elements, and let $V$ be the subspace of $(R \otimes_D)^k$ spanned by this orbit. We note that
$$
V/( 1- t) = V/W,
$$
where $W$ is the subspace spanned by the set
\begin{equation} \label{eq:spanning set for W}
\{ b_1 + (-1)^k b_2, \ b_2 + (-1)^k b_3, \ \ldots, \  b_{p - 1} + (-1)^k b_p, \ (-1)^k b_1 + b_p \},
\end{equation}
assuming an appropriate ordering of the elements of the orbit. By \cref{lem:counting basis elements of period p}, it suffices to show that $\dim V/W = 0$  if $k$ is even and $p$ is odd, and $\dim V/W = \frac{1}{p} \dim V = 1$ (i.e.\ $\dim W = p - 1$) otherwise.

First suppose that $k$ is even and $p$ is odd. Then in $V/W$, we have
$$
[b_1] = -[b_2] = [b_3] = \ldots = (-1)^p [b_1] = -[b_1].
$$
Since the characteristic of the field $K$ is zero (and thus, in particular, not equal to $2$), we can deduce that $[b_1] = 0$, and hence $V/(1 - t) = 0.$

Next we suppose that $k$ and $p$ are both even. Then the first $p - 1$ elements of the spanning set of $W$ given in \eqref{eq:spanning set for W} are linearly independent, while the last element can be written as a linear combination of the others as follows:
$$
b_1 + b_p = \sum_{i = 1}^{p - 1} (-1)^{i - 1} (b_i + b_{i + 1})
$$
This shows that $\dim W = p -1$.

Finally, suppose $k$ is odd. Then the first $p - 1$ elements of the spanning set in \eqref{eq:spanning set for W} are again linearly independent, and we have
$$
-b_1 + b_p = -\sum_{i = 1}^{p  - 1} (b_i - b_{i + 1}),
$$
which again shows that $\dim W = p - 1$, as desired.
\end{proof}

\subsection{Proof of Igusa's formula}

\Cref{lem:dimension of quotient space in complex} allows us to compute the graded Euler characteristic of $\HC_*(\Lambda, R)$. In order to connect the Euler characteristic to the logarithm of $\det C_\Lambda(x)$ and prove formula \eqref{eq:igusa's formula - initial statement}, we need the following result.

\begin{prop}[{\cite[Proposition 2.1]{IGUSA1992101}}]\label{prop:log of determinant equals trace series}
    $$
    \log \det C_\Lambda(x) = \sum_{k = 1}^\infty (-1)^{k + 1} \frac{\tr(E(x)^k)}{k}.
    $$
\end{prop}
\begin{proof}
This can be shown using the same proof as in \cite{IGUSA1992101}, and we refer the reader to that article for more details. However, the proof in said article only works if two conditions are satisfied: both sides of the equality must be well defined power series, and the eigenvalues of $E(0)$ must all be zero. We therefore check that these conditions are satisfied under our assumptions.

To see that the left hand side of the equality is well defined, note that the constant term of $C_\Lambda(x)$ is $C_{\Lambda_0}$, the (ungraded) Cartan matrix of $\Lambda_0$. Since $\Lambda_0$ is given by an acyclic quiver with relations, the determinant of $C_{\Lambda_0}$ is equal to $1$. 
Then the constant term of $\det C_\Lambda(x)$ is $1$, so the power series
$$
\log\det C_\Lambda(x) = \sum_{k = 1}^\infty (-1)^{k + 1} \frac{(\det C_\Lambda(x) - 1)^k}{k}
$$
is well defined. 

For the right hand side of the equality, we recall from \cref{lem:dimension of homogeneous component of tensor product} and \eqref{eq:iterated tensor product} that the coefficient of $x^l$ in $\tr(E(x)^k)$ is the dimension of the degree $l$ component of
$$
\bigoplus_{i_1, \ldots, i_k} e_{i_1} R e_{i_2} \otimes \cdots \otimes e_{i_k} R e_{i_1}.
$$
Any homogeneous element of $e_{i_1} R e_{i_2} \otimes \cdots \otimes e_{i_k} R e_{i_1}$ of degree $l$ may be written as a linear combination of elementary tensors of the form $b_1 \otimes \cdots \otimes b_k$, where $b_j \in e_{i_j} R e_{i_{j + 1}}$ (where we set $i_{k+1} =i_1$) are homogeneous elements such that $\sum_j |b_j| = l$. 
Furthermore, if an element $b_j$ is of degree zero, it may be assumed to be a nontrivial path from $e_{i_j}$ to $e_{i_{j +1}}$ in the quiver of $\Lambda_0$, since such an element $b_j$ must be contained in $R_0 = \rad (\Lambda_0)$.
Since the quiver of $\Lambda_0$ is acyclic, there is an upper bound $N$ on the lengths of paths in the quiver. If $k$ is sufficiently large compared to $l$, there must exist some consecutive indices $j, j + 1, \ldots, j + N$ such that the elements $b_j, \ldots, b_{j + N}$ are all of degree zero. But then these elements define a path in the quiver of $\Lambda_0$ whose length is at least $N + 1$, which is impossible. Thus the degree $l$ component of $\bigoplus_{i_1, \ldots, i_k} e_{i_1} R e_{i_2} \otimes \cdots \otimes e_{i_k} R e_{i_1}$ vanishes for sufficiently large values of $k$, and hence the monomial $x^l$ only appears in a finite number of terms of the infinite sum $\sum_{k = 1}^\infty (-1)^{k + 1} \frac{\tr(E(x)^k)}{k}$, which shows that the sum is a well defined power series. 

Finally, since $\Lambda_0$ is given by an acyclic quiver, the vertices of the quiver can be ordered in such a way that the Cartan matrix of $\Lambda_0$ is upper triangular, with ones on the main diagonal. Then $E(0) = C_{\Lambda_0} - I_{n}$ is strictly upper triangular, hence its eigenvalues are zero.
\end{proof}

The last result needed for the proof of formula \eqref{eq:igusa's formula - initial statement} is the following number-theoretic lemma, which is well known according to \cite{IGUSA1992101}. We include a proof for the convenience of the reader.

\begin{lem}[{\cite[Lemma 3.2]{IGUSA1992101}}]
\label{lem: number theoretical formula for even n}

If $n$ is an even positive integer, then
$$
\sum_{d | n} \frac{\mu(d)}{d} = -\sum_{\mathclap{\substack{d | n \\ d\text{ even}}}} \frac{\mu(d)}{d}.
$$

\end{lem}
\begin{proof}
For any positive integer $k$, we let $F(k) = \sum_{d | k} \frac{\mu(d)}{d}$, which we note is a multiplicative arithmetic function of $k$. (For example, it can be seen that $F = \frac{1}{\operatorname{id}(-)} \cdot(\mu * \operatorname{id})$, where $*$ denotes Dirichlet convolution; see \cref{sec:dirichlet convolution}.) Since $\mu(d) = 0$ if $d$ is divisible by the square of a prime, it is enough to prove the equality in the case where $n$ is a square-free number. So let $n$ be a square-free even number, and write $n = 2m$ for an odd number $m$. Then we have
\begin{align*}
    \sum_{d | n} \frac{\mu(d)}{d} &= F(n) \\
    &= F(2m) \\
    &= F(2)F(m) \\
    &= \frac{1}{2} \sum_{d' | m} \frac{\mu(d')}{d'} \\
    &= -\sum_{d' |m} \frac{\mu(2d')}{2d'},
\end{align*}
where the last equality uses the fact that $\mu(d') = -\mu(2d')$ since $d'$ is odd. Now we note that even divisors of $n = 2m$ are precisely of the form $2d'$ for $d'$ a divisor of $m$, so we get
$$
 -\sum_{d' |m} \frac{\mu(2d')}{2d'} = -\sum_{\mathclap{\substack{d | n \\ d \text{ even}}}} \frac{\mu(d)}{d},
$$
which completes the proof.
\end{proof}

We are now ready to prove that Igusa's formula holds under our assumptions.

\begin{thm}[{\cite[Theorem 3.5]{IGUSA1992101}}]
\label{thm:igusa main theorem}
We have the following equalities:
\begin{equation} \label{eq:igusa's main theorem - formula for euler characteristic}
\chi_{\HC_*(\Lambda, R)}(x) = \sum_{k = 1}^\infty \log\det C_\Lambda(x^k)\sum_{d| k} \frac{\mu(d)}{d}
\end{equation}
\begin{equation} \label{eq:igusa's main theorem - formula for logarithm of determinant}
\log\det C_\Lambda(x) = \sum_{k = 1}^\infty \chi_{\HC_*(\Lambda, R)}(x^k)\sum_{d |k} \frac{d \mu(d)}{k}
\end{equation}
In particular, the right hand sides are well defined power series.

\end{thm}
\begin{proof}
In order to show that the right hand sides are well defined, it suffices to check that $\log\det C_\Lambda(x)$ and $\chi_{\HC_*(\Lambda, R)}(x)$ have no constant term. In the case of $\log \det C_\Lambda(x)$, this follows directly from the definition of the logarithm of a power series with constant term 1, and the case of $\chi_{\HC_*(\Lambda, R)}(x)$ follows since $\HC_*(\Lambda, R)$ vanishes in homogeneous degree zero by \cref{lem:relative cyclic homology vanishes in homogeneous degree zero}.

It can be shown that the arithmetic functions
$$F(k) = \sum_{d | k} \frac{\mu(d)}{d}$$
and
$$G(k) = \sum_{d | k} \frac{d \mu(d)}{k}$$
are mutual inverses with respect to Dirichlet convolution (see \cref{sec:dirichlet convolution}). Then by \cref{lem:convolution and inversion for power series}, it suffices to prove that the equality \eqref{eq:igusa's main theorem - formula for euler characteristic} is true.

In the proof of \cref{prop:log of determinant equals trace series}, we saw that the degree $l$ component of $(R \otimes_D )^k$ vanishes if $k$ is sufficiently large compared to $l$, which also holds for $(R \otimes_D)^k /(1 - t)$. Then the graded Euler characteristic of $\HC_*(\Lambda, R)$ is equal to the graded Euler characteristic of \eqref{eq: igusa's complex for relative cyclic homology}, since that complex is bounded and finite dimensional in each homogeneous degree. The graded Euler characteristic of \eqref{eq: igusa's complex for relative cyclic homology} can be computed using \cref{lem:dimension of quotient space in complex},  and an argument given in \cite{IGUSA1992101} using \cref{lem: number theoretical formula for even n} and \cref{prop:log of determinant equals trace series} shows that the resulting power series is equal to the right hand side of \eqref{eq:igusa's main theorem - formula for euler characteristic}.
\end{proof}

\subsection{Proof of the product formula}

\cref{thm:igusa main theorem} concerns the graded Euler characteristic of the relative cyclic homology $\HC_*(\Lambda, R)$, whereas the formula we are interested in using the theorem to prove --- that is, formula \eqref{eq:product formula for coefficients of euler characteristic} --- concerns the graded Euler characteristic of the reduced cyclic homology $\overline{\HC}_*(\Lambda)$.  We therefore require the following result.

\begin{lem} \label{lem:reduced vs relative cyclic homology}
The relative cyclic homology $\HC_*(\Lambda, R)$ is isomorphic to the reduced cyclic homology $\overline{\HC}_*(\Lambda)$.
\end{lem}
\begin{proof}
Recall from \cref{sec:cyclic homology} that relative cyclic homology fits into a long exact sequence:
$$
\cdots \to \HC_k(\Lambda, R) \to \HC_k(\Lambda) \to \HC_k(D)  \to \HC_{k - 1}(\Lambda, R) \to \cdots
$$
Since $D \cong K^{n}$ for $n = \lvert \Lambda_0 \rvert$, we have $\HC_k(D) \cong \HC_k(K)^{n}$, which vanishes for odd $k$ (see e.g.\ page 59 of \cite{loday}). Hence, the long exact sequence above yields the following bounded exact sequence for any $i$.
\begin{equation*} \label{eq:bounded exact sequence for relative cyclic homology}
0 \to \HC_{2i}(\Lambda, R) \to \HC_{2i}(\Lambda) \to \HC_{2i}(D) \to \HC_{2i - 1}(\Lambda, R) \to \HC_{2i - 1}(\Lambda) \to 0
\end{equation*}
The algebra $D$ is concentrated in homogeneous degree zero, so $\HC_{2i, l}(D)$ vanishes for $l > 0$. Thus we can see from the exact sequence above that $\HC_{*, l}(\Lambda, R)$ is isomorphic to $\HC_{*, l}(\Lambda)$ --- and hence also to $\overline{\HC}_{*, l}(\Lambda)$, by \eqref{eq:value of reduced cyclic homology} --- for $ l > 0$. To complete the proof, it now suffices to observe that $\HC_*(\Lambda, R)$ and $\overline{\HC}_*(\Lambda)$ both vanish in homogeneous degree zero, by \cref{lem:relative cyclic homology vanishes in homogeneous degree zero} and \eqref{eq:value of reduced cyclic homology}.
\end{proof}

We can now prove the formula we will use to deduce the coefficients of $\chi_{\overline{\HC}_*(\Lambda)}(x)$.

\begin{thm}[See \cite{Etingof-Ginzburg'07}, but also \cite{Etingof-Eu'06}]\label{thm: etingof-ginzburg}   
Let $\chi_{\overline{\HC}_*(\Lambda)}(x) = \sum_{k \geq 1} a_k x^k$.
Then  
\begin{equation} \label{eq:product formula (statement in theorem)}
\prod\limits_{k=1}^\infty(1-x^k)^{-a_k}=\prod\limits_{s=1}^\infty\det C_\Lambda(x^s).
\end{equation}
\end{thm}
\begin{proof}
It is sufficient to show that the logarithms of both sides of the equality are equal. 
Hence, we compute
\begin{align*}
\log \prod_{k = 1}^{\infty}(1 - x^k)^{-a_k} 
& = 
- \sum_{k = 1}^{\infty}a_k \log(1 - x^k)\\
& = 
\sum_{k = 1}^{\infty}a_k \sum_{n = 1}^{\infty} \frac{1}{n} x^{kn}.
\end{align*}
For the right hand side, we compute
\begin{align*}
\log \prod_{s = 1}^{\infty}\det C_\Lambda(x^s)
& = 
\sum_{s = 1}^{\infty}\log \det C_\Lambda(x^s) \\
& = 
\sum_{s = 1}^{\infty} \sum_{m = 1}^{\infty} \frac{1}{m} \sum_{d\lvert m} d\mu(d) \sum_{k = 1}^{\infty} a_k x^{kms} \\
& = 
\sum_{n = 1}^{\infty} \sum_{m \lvert n} \frac{1}{m} \sum_{d\lvert m} d\mu(d) \sum_{k = 1}^{\infty} a_k x^{kn},
\end{align*}
where we use \cref{thm:igusa main theorem}
in the second equality. 
Since convergence is absolute for small values of $x$, we can reverse the order of summation, and thus it is enough to show that 
\[
\frac{1}{n} = F(n) := \sum_{m\lvert n } \frac{1}{m} \sum_{d\lvert m} d\mu(d).
\]
For this, we proceed by showing that $F$ is a multiplicative arithmetic function and by computing its values on powers of prime numbers. To see that $F$ is multiplicative, we first note that the arithmetic function
$$
G(m) = \sum_{d | m} d \mu(d)
$$
is the Dirichlet convolution (see \cref{sec:dirichlet convolution}) of the multiplicative function $\operatorname{id}\cdot \mu$ with the constant function $1$. Then $F$ is equal to the Dirichlet convolution $1 * (\frac{1}{\operatorname{id}} \cdot G)$. Since the Dirichlet convolution of two multiplicative functions is again multiplicative, it follows that $F$ is a multiplicative function.

We now use induction to verify that 
\[
\frac{1}{p^m} = F(p^m)
\]
holds for a prime number $p$ and a natural number $m \ge 0$. 
Note that the case $m = 0$ holds since 
\[
F(1) = \frac{1}{1} \cdot 1 \cdot \mu(1) = 1,
\]
as $\mu(1) = 1$.

For the inductive step, we assume that $F(p^m) = \frac{1}{p^{m}}$ for some $m \ge 0$, and  we compute 
\begin{align*}
\sum_{k\lvert p^{m + 1}} \frac{1}{k}\sum_{d\lvert k} d\mu(d) 
& = 
\sum_{k\lvert p^{m}} \frac{1}{k}\sum_{d\lvert k} d\mu(d) 
+ \frac{1}{p^{m+1}}\sum_{d\lvert p^{m+1}}d \mu(d)\\
& = 
\frac{1}{p^m}
+ \frac{1}{p^{m+1}}\sum_{d\lvert p^{m+1}}d \mu(d)
\end{align*}
and note that we are done if we can show that the rightmost term equals $\frac{1}{p^{m+1}}(1-p)$, which it indeed does as 
\begin{align*}
\sum_{d\lvert p^{m+1}}d \mu(d)\ 
& = 
1 + p \mu(p) + p^2 \mu(p^2) + \cdots + p^{m+1}\mu(p^{m+1}) \\
& = 1 - p + 0 + \cdots + 0\\
& = 1 - p,
\end{align*}
since $\mu(n) = 0$ if $n$ is divisible by a square of a prime. 

Finally, writing $n$ as a product of powers of distinct primes $n = \prod_i p_i^{e_i}$, we see that 
\begin{align*}
F(n) = \prod_i F(p_i^{e_i}) = \prod_i \frac{1}{p_i^{e_i}} = \frac{1}{n},
\end{align*}
and hence the logarithms of the two sides of \eqref{eq:product formula (statement in theorem)} are indeed equal as claimed.
\end{proof}

\section{On Frobenius algebras, $d$-hereditary algebras, and stable Hochschild (co)homology}\label{sec: On Frobenius algebras and d-hereditary algebras}
In this section, we begin by reviewing Frobenius algebras, including graded Frobenius algebras, Calabi--Yau Frobenius algebras as in \cite{Eu-Schedler'09}, and twisted periodic algebras. 
Following this, we recall stable Hochschild (co)homology and some formulas that hold for the Hochschild (co)homology of periodic CY Frobenius algebras.
We then review twisted fractionally Calabi--Yau algebras as in \cite{HI11b} before recalling the $d$-representation finite algebras introduced by Iyama and others \cite{Iyama-Oppermann} and their higher preprojective algebras.
Finally, we prove \cref{iprop: graded version of Happel's result} and \cref{iprop: HH_* bounds on degrees} by using Theorem 1.2 of \cite{AO14} in combination with some of the ideas in section 3.1 and 3.2 of \cite{AIR15}.

\subsection{Frobenius algebras}\label{subsection: Frobenius algebras}
Recall that if $M$ is a $\Lambda$-module, then \textit{the syzygy of} $M$ is denoted by $\Omega M$ and may be defined as the kernel of the projective cover of $M$. Also recall that this definition is extended inductively by setting $\Omega^{m + 1}M := \Omega(\Omega^m M)$ for  a natural number $m$.
\textit{Cosyzygies of} $M$ are denoted by $\Omega^{-i} M$ for $i$ a natural number, and are defined dually.

Recall that $\Lambda$ is a \textit{Frobenius algebra} if there is an isomorphism of bimodules $D\Lambda \cong {}_{1}\Lambda_{\mu}$ for $\mu$ an algebra automorphism of $\Lambda$, in which case the latter is called the \textit{Nakayama automorphism of} $\Lambda$; see e.g.\ \cite{Erdmann-Skowronski'08}. 
Recall moreover that every Frobenius algebra is selfinjective and finite dimensional, and that every basic selfinjective finite dimensional algebra is Frobenius. 

We now want to review a graded version of the above.
Hence, recall that a positively graded algebra $\Lambda$ is a \textit{graded Frobenius algebra of highest degree} $a$ if $D\Lambda \cong {}_{1}\Lambda_{\mu} \langle -a\rangle $ for $\mu$ a graded algebra automorphism of $\Lambda$. 
In this case, $\Lambda$ is also Frobenius in the usual sense and $\mu$ is its Nakayama automorphism. 
Note that twisting by the Nakayama automorphism commutes with forming (co)syzygies. 
If $\mu$ can be chosen to be the identity, we say that $\Lambda$ is a \textit{graded symmetric algebra of highest degree} $a$.

Trivial extensions provide a crucial example of the latter notion: Recall that given a finite dimensional algebra $A$, the \textit{trivial extension of} $A$ is defined to be $\Delta(A) := A \oplus D(A)$ as a vector space, and we can make $\Delta(A)$ an $\mathbb{N}$-graded symmetric algebra of highest degree $1$ by putting $A$ in degree $0$, putting $D(A)$ in degree $1$, and endowing $\Delta(A)$ with the multiplication $(a,f)\cdot (b, g) := (ab, ag + fb)$ for $(a,f), (b,g) \in \Delta(A)$. 
Note that one can also think of a trivial extension as a quotient of a tensor algebra over $A$ in the following way: $\Delta(A) \cong T_A D(A)/\langle D(A)^{\otimes_A 2} \rangle$.

Note now that for the rest of this subsection we only consider the graded case since we can usually recover the ungraded case from the graded one by considering an algebra $\Lambda$ with a trivial grading. 

When $\Lambda$ is selfinjective, $\gr \Lambda$ is a Frobenius category and one can thus obtain a triangulated category in the form of \textit{the stable category of finitely generated graded $\Lambda$-modules} $\stgrmodu \Lambda$, i.e.\ the category obtained from $\gr \Lambda$ by taking the quotient by the ideal of morphisms that factor through the projective objects. 
Recall that the shift functor of $\stgrmodu \Lambda$ as a triangulated category is given by the co-syzygy functor $\Omega^{-1}$.
Moreover, also recall that while
\[
\Hom_{\stgrmodu \Lambda} (\Omega^{i}X,  Y \langle j \rangle) \cong \Ext_{\gr \Lambda}^{i}(X, Y \langle j \rangle)
\]
holds for $i > 0$, for $i = 0$ one has that $\Hom_{\stgrmodu \Lambda} (X, Y \langle j \rangle)$ is a quotient of $\Hom_{\gr \Lambda} (X, Y \langle j \rangle)$; see e.g.\ \cite{Buchweitz'86}. 
Note, moreover, that one can consider this formula for negative values of $i$, which we will do for Hochschild cohomology. 
Indeed, by using that $D\Tor_i^{\Lambda}(M,DN) \cong \Ext^{i}_{\Lambda}(M,N)$ holds when working with a $K$-algebra $\Lambda$ and extending the usual definition of $\otimes$ and $\Tor$ via the formula above, one obtains non-positive versions of $\Tor$ and $\Ext$ given by, respectively, 
\[
\stTor^{\Lambda}_i(M, N) := D\Hom_{\stmodu \Lambda}(\Omega^{i}M, D(N))
\]
and 
\[
\stExt_{\Lambda}^{i}(M, N) := \Hom_{\stmodu \Lambda}(\Omega^{i}M, N)
\]
for $M, N \in \stmodu \Lambda$.

Consequently, one obtains non-positive versions of both Hochschild homology and cohomology given by, respectively, 
\[
\stHH_{i,j}(\Lambda) := D\Hom_{\stgrmodu \Lambda^{e}} (\Omega^{i}_{\Lambda^{e}} \Lambda, D(\Lambda) \langle j \rangle) 
\]
and 
\[
\stHH^{i,j}(\Lambda) := \Hom_{\stgrmodu \Lambda^{e}} (\Omega^{i}_{\Lambda^{e}} \Lambda, \Lambda \langle -j \rangle).
\]

Finally, recall that $\stgrmodu \Lambda$ is also endowed with a Serre functor given by $\Omega(-) \otimes_\Lambda D\Lambda \cong \Omega(-)_{\mu} \langle -a \rangle$; see e.g.\ \cite{ARS97} and Proposition I.2.3 of  
\cite{Reiten-VanDenBergh'02}. 

\subsection{Twisted periodic algebras}
We now recall a key class of selfinjective algebras.
\begin{defin}[{\cite[Definition 3.1]{Chan-et-al}}] 
\begin{enumerate}[(i)]
    \item A $\Lambda$-module $M$ is \textit{periodic }if there is some integer $d > 0$ such that $\Omega^d M \cong M$ holds in $\modu \Lambda$.
    \item The algebra $\Lambda$ is \textit{(bimodule)} \textit{periodic} if it is periodic as a $\Lambda^{e}$-module, i.e.\ if there is some $d > 0$ such that $\Omega^{d}_{\Lambda^{e}} \Lambda \cong \Lambda$ holds in $\modu \Lambda^{e}$. In this case, $\Lambda$ is called $d$-periodic.
    \item The algebra $\Lambda$ is \textit{twisted (bimodule) periodic} if there is some $d > 0$ and an algebra automorphism $\phi$ of $\Lambda$ such that $\Omega^{d}_{\Lambda^{e}} \Lambda \cong {}_{1}\Lambda_{\phi}$ holds in $\modu \Lambda^{e}$. In this case, $\Lambda$ is called twisted $d$-periodic and $\phi$ is said to be the associated twist. 
\end{enumerate}
\end{defin} 

Note that it has been conjectured by Erdmann and Skowronsk\'{i} that all twisted periodic algebras are in fact periodic; see e.g.\ \cite{Erdmann-Skowronski'08}.

\subsection{Calabi--Yau Frobenius algebras and Hochschild (co)homology}\label{subsection: CY Frobenius and HH^*}
We now recall a class of algebras that are in some sense 
as close as possible to being Calabi--Yau in the sense of \cite{Gin06} while also being Frobenius.
\begin{defin}[{\cite[Definition 2.3.6]{Eu-Schedler'09}}]\label{def: CY Frobenius}
A Frobenius algebra $\Lambda$ is \textit{Calabi--Yau Frobenius of dimension} $m$ if $\Lambda$ regarded as a bimodule satisfies $\Omega_{\Lambda^{e}}^{m+1}\Lambda \cong \Hom_{\Lambda^{e}}(\Lambda, \Lambda^{e})$ in $\stmodu \Lambda^{e}$ for some $m \in \mathbb{Z}$.
\end{defin}

The following is essentially part (ii) of Lemma 2.1.35 of  \cite{Eu-Schedler'09} adapted to our setup. This result thus translates the definition above into an equivalent one involving the Nakayama automorphism of $\Lambda$. 
We include a proof for the convenience of the reader. 
\begin{prop}[See {\cite[Lemma 2.1.35 (ii)]{Eu-Schedler'09}}]\label{prop: CY Frob characterization}
Let $\Lambda$ be a Frobenius algebra and let $m$ be an integer not equal to $-1$.
Then $\Lambda$ is Calabi--Yau Frobenius of dimension $m$ if and only if $\Omega^{-m-1}_{\Lambda^{e}}\Lambda \cong D\Lambda \cong {}_{1}\Lambda_{\mu}$ holds in $\stmodu \Lambda^{e}$.
\end{prop}
\begin{proof}
Note that 
\begin{align*}
D\Hom_{\Lambda^{e}}(\Lambda, \Lambda^{e}) 
& \cong \Lambda \otimes_{\Lambda^{e}} D\Lambda^{e} \\
& \cong D\Lambda \otimes_\Lambda \Lambda \otimes_{\Lambda} D\Lambda \\
& \cong {}_{1}\Lambda_{\mu} \otimes_\Lambda \Lambda \otimes_{\Lambda} {}_{1}\Lambda_{\mu} \\
& \cong {}_{1}\Lambda_{\mu^{2}}
\end{align*}
so that 
\begin{align*}
\Hom_{\Lambda^{e}}(\Lambda, \Lambda^{e}) 
& \cong D({}_{1}\Lambda_{\mu^{2}})\\
& \cong D({}_{\mu^{-2}}\Lambda{}_{1})\\
& \cong {}_{1}D(\Lambda)_{\mu^{-2}}\\
& \cong {}_{1}\Lambda_{\mu^{-1}}.
\end{align*}
\end{proof}

\begin{remark}
Note that the definition in \cite{Eu-Schedler'09} requires only that the isomorphism holds in $\stmodu \Lambda^{e}$. 
However, since $\Lambda$ is for us a finite dimensional algebra over a field, the isomorphisms in \cref{def: CY Frobenius} and in \cref{prop: CY Frob characterization} will lift to ones in $\modu \Lambda^{e}$ as long as $\Lambda$ considered as a bimodule does not have any bimodule projective summands.
If $\Lambda$ is a ring-indecomposable algebra, i.e.\ its identity is the only non-zero central idempotent, then this only happens if $\Lambda$ is separable, in which case $\Lambda$ would be semisimple.
Consequently, if $\Lambda$ is a ring-indecomposable non-semisimple Calabi--Yau Frobenius algebra of dimension $m \geq 0$, then $\Lambda$ is twisted periodic with twist given by $\mu^{-1}$, i.e.\ the inverse of the Nakayama automorphism of $\Lambda$.   
\end{remark}

Since the Serre functor of $\stgrmodu \Lambda$ is given by $\Omega(-) \otimes_\Lambda \Lambda_{\mu} \langle -a\rangle$ whenever $\Lambda$ is a graded Frobenius algebra of highest degree $a$, one obtains some particularly nice duality formulas involving Hochschild homology and cohomology:

\begin{prop}[See {\cite[Theorem 2.3.27]{Eu-Schedler'09}}] \label{prop: duality formulas}
Let $\Lambda$ be Calabi--Yau Frobenius for some integer $m \geq -1$. 
We then have 
\[
\stHH^{i}(\Lambda) \cong \stHH_{m - i}(\Lambda) 
\]
and 
\[
\stHH^{i}(\Lambda) \cong D\stHH^{2m + 1 - i}(\Lambda).
\]

If we also assume that $\Lambda$ is graded Frobenius of highest degree $a$ and that $$\Omega^{m+1} \Lambda \cong {}_{1}\Lambda_{\mu^{-1}} \langle a + 1\rangle,$$ then we have that 
\[
\stHH^{i,j}(\Lambda) \cong \stHH_{m - i, j+1}(\Lambda) 
\]
and 
\[
\stHH^{i,j}(\Lambda) \cong D\stHH^{2m + 1 - i,-j-2}(\Lambda).
\]
\end{prop}
\begin{proof}
Note that it suffices to show the graded versions since the ungraded ones follow by assuming that $\Lambda$ is graded trivially. 
Moreover, note that we only show the first of the graded isomorphisms since showing the second one is similar.

Hence, we begin by observing that since $\Lambda$ is graded Frobenius of highest degree $a$ and with Nakayama automorphism $\mu$, we have that $\Lambda^{e}$ is graded Frobenius of highest degree $2a$ and with Nakayama automorphism $\mu^{-1} \otimes \mu$.
Consequently, we can use that $\stgrmodu \Lambda^{e}$ has a Serre functor and compute as follows.

\begin{align*}
\stHH^{i,j}(\Lambda) 
& = \Hom_{\stgrmodu \Lambda^{e}}(\Omega^{i}_{\Lambda^{e}}\Lambda, \Lambda \langle {-j} \rangle)\\
& \cong D\Hom_{\stgrmodu \Lambda^{e}}(\Lambda \langle {-j}  \rangle, \Omega^{i+1}_{\Lambda^{e}} {}_{1}\Lambda_{\mu^{2}} \langle -2a \rangle)\\
& \cong D\Hom_{\stgrmodu \Lambda^{e}}(\Omega^{-i - 1}_{\Lambda^{e}}\Lambda ,  {}_{1}\Lambda_{\mu^{2}} \langle {-2a + j}  \rangle)\\
& \cong 
D\Hom_{\stgrmodu \Lambda^{e}}(\Omega^{- i - 1}_{\Lambda^{e}}\Lambda ,  \Omega^{-(m + 1)}_{\Lambda^{e}}{}_{1}\Lambda_{\mu} \langle {-a + 1 + j} \rangle)\\
& \cong D\Hom_{\stgrmodu \Lambda^{e}}(\Omega^{m- i}_{\Lambda^{e}}\Lambda \langle {- j - 1} \rangle,  {}_{1}\Lambda_{\mu} \langle -a \rangle)\\
& \cong D\Hom_{\stgrmodu \Lambda^{e}}(\Omega^{m- i}_{\Lambda^{e}}\Lambda \langle {- j - 1} \rangle,  D(\Lambda ))\\
& = \stHH_{m-i,{j + 1}}(\Lambda)
\end{align*}

Here, the first isomorphism uses that $\stgrmodu \Lambda^{e}$ has a Serre functor, the third uses that $\Lambda$ is Calabi--Yau Frobenius of dimension $m$, the fifth holds since $\Lambda$ is graded Frobenius of highest degree $a$, whereas the second and fourth are  using that $\Omega$ and $\langle 1 \rangle$ are equivalences on $\stgrmodu \Lambda^{e}$.
\end{proof} 

If we also assume that the Nakayama automorphism is of finite order, more can be said. 

\begin{prop}\label{prop: HH formula for periodic + CY Frob}
Let $\Lambda$ be Calabi--Yau Frobenius for some integer $m \geq -1$ and assume that $\mu^n = 1$ holds for some positive integer $n$.
We then have 
\[
\stHH_{i}(\Lambda) \cong D \stHH_{n(m + 1) - 1 - i}(\Lambda).
\]

If we also assume that $\Lambda$ is graded Frobenius of highest degree $a$ and that $\Omega^{m+1}_{\Lambda^{e}} \Lambda \cong {}_{1}\Lambda_{\mu^{-1}} \langle a + 1\rangle$, then we have that 
\[
\stHH_{i,j}(\Lambda) \cong D \stHH_{n(m + 1) - 1 - i,n(a+1) - j }(\Lambda).
\]
\end{prop}
\begin{proof}
As before, we only show the claim in the graded case. 
For this, we proceed as follows:
\begin{align*}
\stHH_{i,j}(\Lambda) & = D\Hom_{\stgrmodu \Lambda^{e}}(\Omega^{i}_{\Lambda^{e}} \Lambda, D(\Lambda) \langle j \rangle)\\
& \cong D\Hom_{\stgrmodu \Lambda^{e}}(\Omega^{i}_{\Lambda^{e}} \Lambda, \Lambda_{\mu} \langle -a+j \rangle)\\
& \cong \Hom_{\stgrmodu \Lambda^{e}}(\Lambda_{\mu} \langle -a+j \rangle, \Omega^{i+1}_{\Lambda^{e}} \Lambda_{\mu^2}\langle -2a\rangle)\\
& \cong \Hom_{\stgrmodu \Lambda^{e}}(\Lambda \langle j \rangle, \Omega^{i+1}_{\Lambda^{e}} \Lambda_{\mu}\langle -a\rangle)
\end{align*}

Now, note that by our assumptions on $\Lambda$ and $\mu$, we have 
\[
\Lambda \cong \Lambda_{\mu^{n}} \cong \Omega^{-n(m+1)}_{\Lambda^{e}} \Lambda \langle n(a+1) \rangle.
\]
Consequently, we deduce that 
\begin{align*}
\stHH_{i,j}(\Lambda) 
& \cong \Hom_{\stgrmodu \Lambda^{e}}(\Lambda \langle j \rangle, \Omega^{i+1}_{\Lambda^{e}} \Lambda_{\mu}\langle -a\rangle)\\
& \cong \Hom_{\stgrmodu \Lambda^{e}}(\Lambda \langle j \rangle, \Omega^{i+1 - n(m+1)}_{\Lambda^{e}} \Lambda_{\mu}\langle -a + n(a+1)\rangle)\\
& \cong \Hom_{\stgrmodu \Lambda^{e}}(\Omega^{n(m+1) -i-1}_{\Lambda^{e}}\Lambda \langle -n(a+1) + j \rangle,  \Lambda_{\mu}\langle -a \rangle)\\
& \cong \Hom_{\stgrmodu \Lambda^{e}}(\Omega^{n(m+1) -i-1}_{\Lambda^{e}}\Lambda \langle -n(a+1) + j \rangle,  D\Lambda)\\
& \cong D\stHH_{n(m+1) - 1 - i, n(a+1) - j }(\Lambda).
\end{align*}
\end{proof}

\subsection{Twisted fractionally Calabi--Yau algebras}
The algebras this subsection is named for were introduced by Herschend and Iyama in \cite{HI11b}. 
Let $A$ be a finite dimensional algebra and 
recall that when the global dimension of $A$ is finite, $\D^b(\modu A)$ --- i.e.\ the bounded derived category of $A$ --- has a Serre functor given by the derived Nakayama functor \mbox{$\nu(-) = - \otimes^{\mathbb{L}}_{A} DA$}. 

Observe that a $K$-algebra automorphism $\phi$ of $A$ induces an autoequivalence of $\D^b(A)$ given by 
\[
- \otimes_{A}^{\mathbb{L}} A_{\phi} \colon \D^b (\modu A) \rightarrow \D^b (\modu A).
\]

\begin{defin}\label{def: (twisted) fractionally CY }
Let $A$ be a finite dimensional $K$-algebra of finite global dimension. 
Then $A$ is said to be \textit{twisted fractionally Calabi--Yau} (or \textit{twisted} $\frac{m}{\ell}$-CY) if there exists an isomorphism 
\[
\nu^{\ell}(-) \cong (-) \otimes_{A}^{\mathbb{L}}A_\phi[m]
\]
of functors for some integers $\ell \neq 0$ and $m$ and a $K$-algebra automorphism $\phi$ of $\Lambda$. 
When $\phi$ can be chosen equal to the identity automorphism of $A$, we say that $A$ is \textit{fractionally CY} (or $\frac{m}{\ell}$-\textit{CY}).
\end{defin}

\subsection{$d$-representation finite algebras}
The algebras mentioned in the title of this subsection were introduced in \cite{Iya11} and \cite{IO13}. 
Roughly speaking, one can think of $d$-representation finite algebras as algebras of global dimension $d$ for which the Serre functor of their bounded derived categories behaves in a fashion that generalizes the behaviour of the Serre functor on the bounded derived category of representation finite hereditary algebras.\footnote{
Note that there is a corresponding notion of $d$-representation infinite algebras that generalizes representation infinite hereditary algebras; see \cite{HIO14}.
}
In particular, these two classes of algebras coincide for $d = 1$. 
We now recall some definitions and basic results from \cite{Iya11, Iyama-Oppermann} and \cite{HI11b}.
In doing so, we establish the necessary background for some applications of \cref{thm: etingof-ginzburg} that we will develop and present in the following sections. 

Throughout the rest of this subsection, we let $A$ be a finite dimensional algebra. 

\begin{defin}\label{def: d-representation finite}
A finite dimensional algebra $A$ is said to be \textit{$d$-representation finite} if $\gldim A \leq d$ and $\modu A$ contains a \textit{$d$-cluster tilting module}, 
i.e.\ a module $M \in \modu A$ such that $\mathcal{M} := \add M$ satisfies 
\begin{align*}
\mathcal{M} 
& = \{X \in \modu A \, \lvert \, \Ext^{i}_A (\mathcal{M}, X) = 0 \, \text{for} \, 1 \leq i \leq d - 1 \, \}\\
& = \{X \in \modu A \, \lvert \, \Ext^{i}_A (X, \mathcal{M}) = 0 \, \text{for} \, 1 \leq i \leq d - 1 \, \}.
\end{align*}
\end{defin}
Since the $\Ext$-vanishing conditions above are vacuous when $d = 1$, one sees that having a $1$-cluster tilting module is equivalent to being representation finite.

Recall that when the global dimension of $A$ is finite, $\D^b(\modu A)$ has a Serre functor given by the derived Nakayama functor \mbox{$\nu(-) = - \otimes^{\mathbb{L}}_{A} DA$}. 
Note that following e.g.\ \cite{HIO14}, we use the notation $\nu_d = \nu(-)[-d]$, where $[1]$ denotes the shift functor of $\D^b(\modu A)$.
Using this, one also has the following criterion for $d$-representation finiteness formulated in terms of the subcategory
\[
\mathcal{U} = \add\{\nu_d^i A \mid i \in \mathbb{Z}\} \subseteq \D^b(\modu A).
\]
    
\begin{thm}[See {\cite[Theorem 3.1]{IO13}}] \label{nrepfindef}
Assume $\gldim A \leq d$. The following are equivalent:
\begin{enumerate}[(i)]
    \item $A$ is $d$-representation finite.
    \item $DA \in \U$.
    \item $\nu \U = \U$.
\end{enumerate}
\end{thm}

Consequently, an algebra $A$ with $\gldim A \leq d$ is $d$-representation finite if and only if for any indecomposable projective $A$-module $P_i$ one has an integer \mbox{$\ell_i \geq 0$} such that $\nu_d^{-\ell_i + 1}(P_i)$ is indecomposable injective.
Note that, following \cite{HI11b}, if there is some integer $\ell$ such that $\ell_i = \ell$ for all $i$, then $A$ is said to be \textit{$\ell$-homogeneous}. 
Recall that the $d$\textit{-Auslander-Reiten translate} of $A$ is defined as $\tau_d := \tau \circ \Omega^{d-1}$. 
In particular, one has that $\tau_d(-) \cong \Hm^0 (\nu_d(-))$.

We recall the following from \cite{HI11b}. 

\begin{thm}[See  {\cite[Theorem 1.1 (a)]{HI11b}}]\label{thm: d-rf implies twisted fr CY} Let $A$ be a finite dimensional algebra that is ring-indecomposable $d$-representation finite. 
Then $A$ is twisted fractionally CY.
\end{thm}

We now recall a simplified version of another result from \cite{HI11b} that will be crucial for one of our applications. 

\begin{prop}[{\cite[Corollary 1.5]{HI11b}} with $k = 2$]
Let $K$ be a perfect field and $\ell$ a positive integer. 
If $A_i$ is $\ell$-homogeneous $d_i$-representation finite for $i = 1,2$, then 
\[
A_1 \otimes_K A_2
\]
is $\ell$-homogeneous $d_1 + d_2$-representation finite. 
\end{prop}

As in the classical theory of hereditary algebras, there is also a notion of (higher) preprojective algebras. 
For a $d$-hereditary algebra $A$, we denote the \textit{$(d+1)$-preprojective algebra of} $A$ by $\Pi_{d+1}A$.
Recall from Lemma 2.13 of \cite{IO13} that
\begin{align*}
\Pi_{d+1}A 
& 
\cong T_A(\tau_d^{-1}A) \\
& 
\cong \bigoplus_{i \geq 0} \tau^{-i}_{d}A.
\end{align*}
Note that one can consider $\Pi_{d+1}A$ to be a graded algebra by putting $(\tau_d^{-1}A)^{\otimes_A^{i}} \cong \tau_d^{-i}A$ in degree $i$, and we call this the \textit{preprojective} or \textit{higher preprojective  grading}.

The following proposition recalls various properties of $(d+1)$-preprojective algebras that we make use of in the sequel. 

\begin{prop}\label{prop: properties of higher preprojective algebras of d-RF algebras}
Let $A$ be a $d$-representation finite algebra. 
Then $\Pi_{d+1}(A)$ is  
\begin{enumerate}[(i)]
    \item finite dimensional by Lemma 2.13 of \cite{IO13};
    \item selfinjective by Corollary 3.4 of \cite{IO13};
    \item graded Frobenius of highest degree $a = \ell-1$ if $A$ is also assumed to be $\ell$-homogeneous by Proposition 2.4 of \cite{HI11b}.
\end{enumerate}
If the base field $K$ is algebraically closed, we also have that $\Pi_{d+1}(A)$ is
\begin{enumerate}[(i)]
  \setcounter{enumi}{3}
    \item $(d+1)$-Calabi--Yau Frobenius by Theorem 1.2 of \cite{AO14} if $A$ is basic or $\ell$-homogeneous; but see also \cite{Dugas'12} and \cite{Grant-Iyama'20}; and 
    \item bimodule stably $\langle 1 \rangle$-twisted $(d+1)$-Calabi--Yau in the sense of \cite{AO14}, and, if $A$ is $\ell$-homogeneous, it satisfies that $\Omega^{d+2}_{\Pi^{e}}\Pi \cong {}_{1}\Pi_{\mu^{-1}} \langle \ell \rangle$.
\end{enumerate}
\end{prop}
\begin{proof}
Let $\Pi := \Pi_{d+1}(A)$.
Part \textit{(i)} and \textit{(ii)} are both immediate from the results cited in the formulation above.

For \textit{(iii)}, we note that $A$ being $\ell$-homogeneous means that the highest degree of any indecomposable projective $\Pi$-module $P$ is equal to $a = \ell - 1$. Consequently, $\socM \Pi \subseteq \Pi_{\ell - 1}$. Moreover, Proposition 2.4 of \cite{HI11b} implies that $\Pi$ is Frobenius. Hence, the result now follows from e.g.\ Lemma 2.4 of \cite{HS}. 

For \textit{(iv)}, this is almost immediate from \cite{AO14} as it only remains to show that $\Pi$ is Frobenius. For the latter, we proceed as follows.
If $A$ is basic, then $\Pi$ is also basic, and \textit{(i)} and \textit{(ii)} together imply that $\Pi$ is Frobenius. Alternatively, if $A$ is $\ell$-homogeneous, then $\Pi$ is Frobenius by \textit{(iii)}.

For \textit{(v)}, we note that Theorem 1.2 \cite{AO14} states that $\Pi_{d+1}(A)$ is bimodule stably $\langle 1 \rangle$-twisted $(d+1)$-Calabi--Yau, meaning, by definition, that $\Omega^{d+2}_{\Pi^{e}} \Pi \cong \Hom_{\Pi^e}(\Pi, \Pi^{e})\langle1\rangle$ holds. 
Hence, \cref{prop: CY Frob characterization} together with \textit{(iii)} above allows us to deduce that $$\Omega^{d+2}_{\Pi^{e}}\Pi \cong {}_{1}\Pi_{\mu^{-1}} \langle \ell \rangle$$ holds in case $A$ is $\ell$-homogeneous.
\end{proof}

\subsection{Higher preprojective algebras, bimodule resolutions, and Hochschild homology}
In this section, we establish some results involving degrees of generators in projective resolutions that allow us to simplify the computation of Hochschild homology for all higher preprojective algebras of $d$-representation finite algebras in a similar fashion to what is done in \cite{Etingof-Eu'06}, \cite{Evans-Pugh'12} and \cite{Morigi'22}.
In particular, as discussed in the introduction in relation to (B), all of \cite{Etingof-Eu'06, Evans-Pugh'12} and \cite{Morigi'22} use that the algebras they work with have explicitly describable graded minimal projective resolutions of the algebras considered as bimodules. 
In some sense, these resolutions are available since the algebras in question are almost Koszul as in \cite{BBK02}. 

Now, while a higher preprojective algebra $\Pi_{d+1}A$ of a $d$-representation finite algebra $A$ is not almost Koszul in general\footnote{Indeed, $\Pi_{d+1}A$ is almost Koszul if and only if $A$ is $d$-representation finite and Koszul by \cite{Grant-Iyama'20}.}, it is nevertheless not unreasonable to expect that the minimal graded projective resolution of $\Pi_{d+1}A$ should be ``nice'' when the latter is considered with the (higher) preprojective grading. Namely, this is the case since a higher preprojective algebra $\Pi_{d+1}A$ is the dual algebra to the trivial extension of $A$, a higher almost Koszul algebra in the sense of \cite{Haugland-Sandoy'26}.
Note thus that for the remainder of this section, we consider $\Pi_{d+1}A$ to be endowed with the (higher) preprojective grading. Also recall that, unless otherwise noted, we work with right modules.

It turns out that it is sufficient for our purposes to know the degrees that the terms of the minimal graded projective bimodule resolution of $\Pi_{d+1}(A)$ are generated in. 
To achieve this, we use one of the main results of \cite{AO14} in combination with some of the ideas in sections 3.1 and 3.2 of \cite{AIR15} to determine the degrees of generators of terms in the minimal projective resolution of the top of $\Pi_{d+1}(A)$. 
By using a graded version of a result from \cite{Hap89}, we can use this to obtain the desired information about the minimal graded projective bimodule resolution of $\Pi_{d+1}(A)$. Note that the relevant result in \cite{AO14} requires that $K$ be algebraically closed, and so we assume this for the remainder of this section.

\begin{prop}\label{prop: degrees of generators of terms of bimodule projective resolution}
Let $A$ be an $\ell$-homogeneous $d$-representation finite algebra. Moreover, let $S := \topM \Pi_{d+1}(A)$ and let $P^{\bullet}(S)$ be the minimal graded projective resolution of $S$.
If we write the integer $i$ as $i = (d+2)q + r$ for integers $q, r$ such that $0 \leq r < d + 2$, then $P^{-i}(S)$ is generated in degree $q\ell$ if $r = 0$ and in degrees $q\ell$ and $q\ell + 1$ otherwise.  
Moreover, we have that $\Omega^{d+2}(S) \cong S \langle \ell \rangle$. 
\end{prop}

\begin{proof} 
Let $\Pi := \Pi_{d+1}(A)$.
We begin by showing the second claim above. 
For this, we recall that \cref{prop: properties of higher preprojective algebras of d-RF algebras} \textit{(v)} entails that $\Omega^{d+2}\Pi \cong {}_{1}\Pi_{\mu^{-1}}$, and so we are done by applying $- \otimes_\Pi S$ and using that $M \otimes_\Pi \Omega^{d+2}_{\Pi^{e}}\Pi \cong \Omega^{d+2}M$ holds in $\stmodu \Pi$ for every $\Pi$-module $M$. Note that the latter claim follows by using that the bimodule projective resolution of a finite dimensional algebra $\Lambda$ is a splicing of split short exact sequences when considered as a complex of left or right modules, and that $M \otimes_{\Lambda} P$ is a right projective $\Lambda$-module whenever $P$ is a projective $\Lambda^{e}$-module.

We now show the first claim. For this, we begin by observing that the second claim implies that it suffices to establish the degrees of the generators of $P^{-i}(S)$ for $0 \leq i \leq d + 1$, and the remainder of the proof thus consists of doing just that by essentially using some of the ideas of section 3.1 and 3.2 of \cite{AIR15}.

Since $\Omega^{d+2} (S)\cong S\langle \ell \rangle$, we can deduce that $P^{-(d+1)}(S)$ must be generated in degree $1$ since $\Pi$ is graded Frobenius of highest degree $\ell - 1$ by \cref{prop: properties of higher preprojective algebras of d-RF algebras} (iii) due to our assumption that $A$ is $\ell$-homogenous. 
Indeed, on the one hand, we see that since $S \langle \ell \rangle$ must lie in the socle of $P^{-(d+1)}(S)$, we can deduce that $P^{-(d+1)}(S)$ must have summands that are generated in degree $1$. On the other hand, since $\Pi$ is selfinjective, the restriction of the differential starting in $P^{-(d+1)}(S)$ to an indecomposable summand of $P^{-(d+1)}(S)$ cannot be a monomorphism, thus implying that no summand of $P^{-(d+1)}(S)$ can be generated in degree different from $1$.

Now observe that, as a consequence of what we have just shown, we have that $P^{-i}(S)$ for $i \leq d$ is generated at most in degrees $0$ and $1$; see e.g.\ Lemma 2.9 \textit{(5)} of \cite{HS}.
Since the degree $0$ part of $P^{\bullet}(S)$ restricts to a minimal projective resolution of $S$ as an $A$-module and $\gldim A = d$, we have that $P^{-i}(S)$ must have generators in degree $0$ for each $0 \leq i \leq d$. 

It now only remains to show that it also has generators in degree $1$ for $1 \leq i \leq d + 1$.
We now proceed to deduce this by using that $(-)^* := \Hom_{\Pi}(-, \Pi)$ sends $P^{[-(d+1), 0]}$ to a complex $Q^{-i} := P^{i -(d + 1)}(S)^*$ giving the first $d+2$ terms of a graded minimal projective resolution of $(S\langle \ell \rangle)^*$ considered as a left $\Pi$-module.

Observe then that 
\begin{align*}
(S\langle \ell \rangle)^* 
& = \Hom_{\Pi}(S\langle \ell \rangle, \Pi) \\
& \cong \Hom_{\Pi}(S\langle 1 \rangle, \Pi\langle -\ell + 1 \rangle) \\
& \cong \Hom_{\Pi}(S_{\mu^{-1}}\langle 1 \rangle, \Pi_{\mu^{-1}}\langle -\ell + 1 \rangle) \\
& \cong \Hom_{\Pi}(S_{\mu^{-1}}\langle 1 \rangle, D\Pi ) \\
& \cong \Hom_{\Pi^{\op}}(\Pi, D(S_{\mu^{-1}}\langle 1 \rangle))\\
& \cong D(S_{\mu^{-1}}\langle 1 \rangle).
\end{align*}
Hence, we have that $Q^{[-(d+1), 0]}$ is a graded minimal (left) projective resolution of $$(S\langle \ell \rangle)^* \cong D(S_{\mu^{-1}} \langle 1 \rangle) \cong D(S)\langle -1 \rangle,$$ where in the last isomorphism we use that $S_{\mu^{-1}} \cong S$ and that $D(-\langle j) \rangle \cong D(-)\langle - j\rangle$.
We deduce then that the arguments we gave above establish that $Q^{-i} = P^{i -(d + 1)}(S)^*$ has generators in degree $-1$ for $0 \leq i \leq d$, thus implying that $P^{-i}(S)$ has generators in degree $1$ for $1 \leq i \leq d +1$. 
\end{proof}

The following is now an easy consequence by using a graded version of Lemma 1.5 of \cite{Hap89}.

\begin{prop}\label{prop: graded version of Happel's result}
Let $A$ be an $\ell$-homogeneous $d$-representation finite algebra and write the integer $i$ as $i = (d+2)q + r$ for integers $q, r$ such that $0 \leq r < d + 2$.  
Then the graded minimal projective bimodule resolution $P^
{\bullet} := P^{\bullet}(\Pi_{d+1}(A))$ of $\Pi := \Pi_{d+1}(A)$ satisfies that $P^{-i}$ is generated in degree $q\ell$ if $r = 0$ and in degrees $q\ell$ and $q\ell + 1$ otherwise. 
\end{prop}
\begin{proof}
Note that if we let $S = \topM \Lambda$, we have $\Hom_k(S, S) \cong \topM_{\Lambda^{e}} (\Lambda \otimes_k \Lambda)$. Moreover, recall that the proof of Lemma 1.5 of \cite{Hap89} proceeds by using the isomorphism
\begin{align*}
\Ext^{i}_{\Lambda^{e}}(\Lambda, \Hom_K(S,S))
& \cong \Ext^{i}_{\Lambda}(S, S). 
\end{align*}
This isomorphism can be proven using standard $\otimes$-$\Hom$-adjunction as follows.
\begin{align*}
\RHom_{\Lambda}(M, N) 
& \cong \RHom_{\Lambda}(\Lambda \otimes_{\Lambda^{e}}^{\mathbb{L}} (M \otimes_K \Lambda), N)\\
& \cong \RHom_{\Lambda^{e}}(\Lambda, \RHom_{\Lambda}(M \otimes_K \Lambda, N))\\
& \cong \RHom_{\Lambda^{e}}(\Lambda, \RHom_{K}(M, \RHom_\Lambda(\Lambda, N)))\\
& \cong \RHom_{\Lambda^{e}}(\Lambda, \RHom_{K}(M, N))\\
& \cong \RHom_{\Lambda^{e}}(\Lambda, \Hom_K(M,N))
\end{align*}
Consider now the case where $M$ and $N$ are graded $\Lambda$-modules, and observe that all of the isomorphisms in the above are homogeneous of (internal) degree zero. 
Consequently, we deduce that there is an isomorphism as follows.
\begin{align*}
\Ext^{i}_{\Gr \Lambda^{e}}(\Lambda, \Hom_K(S, S)\langle j\rangle)
& \cong \Ext^{i}_{\Gr \Lambda}(S, S\langle j \rangle) 
\end{align*}
Evidently, the conclusion follows by using \cref{prop: degrees of generators of terms of bimodule projective resolution}.
\end{proof}

\begin{remark}
Almost $d$-$T$-Koszul algebras were introduced in \cite{HS} as a common generalization of the almost Koszul algebras of \cite{BBK02} and trivial extensions of $d$-representation finite algebras. 
Similarly to the case of a Koszul or almost Koszul algebra $\Lambda$, there is a notion of a dual algebra $\Lambda^!$. 
In the case of a trivial extension $\Delta(A)$ of a $d$-representation finite algebra $A$, the almost Koszul dual $\Delta(A)^!$ is the higher preprojective algebra $\Pi_{d+1}(A)$.
Now, in forthcoming work by the second author and Johanne Haugland, it is shown that some classes of almost $d$-$T$-Koszul algebras come equipped with certain useful projective resolutions: namely, if an algebra $\Lambda$ is almost $d$-Koszul algebra with respect to $\Lambda_0$, then one can in some cases construct a particularly nice $\Lambda$-projective resolution of $\Lambda_0$.
In the case where $\Lambda \simeq \Pi_{d+1}(A)$ for a $d$-representation finite algebra $A$, the minimal projective resolution of $\Lambda$ considered as a bimodule can be seen to be a ``lift'' of the minimal projective resolution of $\Lambda_0$ considered as a $\Lambda$-module that the second author constructs together with Haugland. 
\end{remark}

We now end this section with a result showing that the degree $q\ell$ part of the resolution above often yields only trivial contributions to the Hochschild homology of $\Pi_{d+1}(A)$.

\begin{prop}\label{prop: HH_* bounds on degrees}
Let $A$ be an $\ell$-homogeneous $d$-representation finite algebra, let $i$ be a non-negative integer, and write $i = (d+2)q + r$ for integers $q, r$ such that $0 \leq r < d + 2$. 
Then $$\HH_{i}(\Pi_{d+1}(A)) = \bigoplus_{j = q\ell}^{q\ell + \ell - 1}\HH_{i, j}(\Pi_{d+1}(A))$$ if $r = 0$  and $$\HH_{i}(\Pi_{d+1}(A)) = \bigoplus_{j = q\ell + 1}^{q\ell +\ell}\HH_{i, j}(\Pi_{d+1}(A))$$ otherwise. 
\end{prop}

\begin{proof}
Let $\Pi := \Pi_{d+1}(A)$.
Since $\Omega^{d+2}_{\Pi^{e}}\Pi \cong \Pi_{\mu^{-1}} \langle \ell \rangle$ holds by \cref{prop: properties of higher preprojective algebras of d-RF algebras} \textit{(v)}, it essentially suffices to prove the statement for $0 \leq i \leq d + 1$. 
If $i = 0$, the claim follows by noting that $P ^{0}(\Pi)$ is generated in degree zero and that $\Pi$ is positively graded of highest degree $\ell -1$.

Now suppose that $1 \le i \le d + 1$. Since the preceding proposition implies that $P ^{-i}(\Pi)$ is then generated in degrees $0$ and $1$, we would be done if we could now show that $\HH_{i,0}(\Pi)$ is trivial.
However, we can see that this does indeed hold by using that $$\HH_{i,0}(\Pi) \cong \Hm^{-i} (P^{\bullet}_{0} \otimes_{\Pi^{e}} \Pi_0),$$ that $\Pi_0 = A$, 
that $P^{\bullet}_{0}$ is a minimal projective bimodule resolution of $A$, 
and that $$P^{\bullet}_{0} \otimes_{\Pi^{e}} \Pi_0 \cong P^{\bullet}_{0} \otimes_{A^{e}} A.$$ 
Indeed, from this we deduce that 
\begin{align*}
\HH_{i,0}(\Pi) & \cong \Hm^{-i} (P^{\bullet}_0 \otimes_{A^{e}} A)\\
& \cong \HH_i(A) \\
& = 0
\end{align*}
for $i > 0$ since $\gldim A < \infty$. 
\end{proof}

\section{Some linear algebra}\label{sec: some linear algebra}

In this section, we first recall some background on Coxeter polynomials of tensor product algebras. 
These results are certainly not new and our exposition is based in part on \cite{Happel'97}. Using these results, we obtain an easy computation of the Coxeter polynomials of $K\mathbb{A}_n \otimes K\mathbb{A}_n$. 

Following this, we study the relationship between the Coxeter polynomial of a finite dimensional algebra $A$ and the graded Cartan determinant of its trivial extension $\Delta (A)$. 
Specializing to the case that $A$ is $d$-representation finite, we show a result involving the graded Cartan determinants of $\Delta(A)$ and $\Pi_{d+1}(A)$ that can be viewed as a ``higher almost Koszul'' version (as in \cite{HS}) of results from \cite{BBK02} and \cite{BGS96}.
\subsection{Coxeter polynomials of tensor products of algebras}
Recall that for a basic finite dimensional algebra $A$ of finite global dimension, the \textit{Coxeter matrix of} $A$ is
$$
\Phi_A = -C_{A}^T C_{A}^{-1},
$$
where $C_{A}$ is the (ungraded) Cartan matrix of ${A}$. The \textit{Coxeter polynomial of} ${A}$,  denoted $\phi_{A}(x)$, is the characteristic polynomial of $\Phi_{A}$.

In order to determine the Coxeter polynomial of a tensor product of two finite dimensional algebras, we need the following two simple results. We do not claim that the results are novel (in particular, similar statements appear in \cite{Happel'97}), but we include proofs for the sake of completeness. 
Note that $K$ is assumed to be perfect for the remainder of this subsection.

\begin{lem} \label{lem:coxeter matrix of tensor product}
    Let ${A}$ and ${B}$ be basic finite dimensional algebras of finite global dimension, and let $\Phi_{A}$ and $\Phi_{B}$ denote their Coxeter matrices. Then the Coxeter matrix of the tensor product ${A} \otimes {B}$ is
$$\Phi_{{A} \otimes {B}} = - \Phi_{A} \otimes \Phi_{B},
$$
where the symbol $\otimes$ on the right hand side denotes the Kronecker product of matrices.
\end{lem}
\begin{proof}
    With respect to some choice of a complete set of primitive orthogonal idempotents $e_1, \ldots, e_n$ of ${A}$, the Cartan matrix of ${A}$ is given by the formula
    $$
    C_{A} = (\dim e_i {A} e_j)_{ij}.
$$
Now by choosing a complete set of primitive orthogonal idempotents $e_1', \ldots, e_m'$ of ${B}$, we obtain a set of idempotents $\{ e_{i_1}\otimes e_{i_2}' \}_{i_1, i_2}$ of ${A} \otimes {B}$. Note in particular that these idempotents are primitive since $K$ is assumed to be perfect.
With respect to this choice of idempotents, the Cartan matrix of ${A} \otimes {B}$ is
$$
C_{{A} \otimes {B}} = (\dim((e_{i_1} \otimes e_{i_2}') ({A} \otimes {B}) (e_{j_1} \otimes e_{j_2}')))_{(i_1, i_2), (j_1, j_2)},
$$
where the tuples $(i_1, i_2)$ and $(j_1, j_2)$ are regarded as a row index and a column index, respectively. Now we note that this matrix may be rewritten as
$$
(\dim(e_{i_1} {A} e_{j_1}) \dim (e_{i_2}' {B} e_{j_2}') )_{(i_1, i_2), (j_1, j_2)},
$$
which is the Kronecker product of $C_{A}$ and $C_{B}$. We thus see that $C_{{A} \otimes {B}} = C_{A} \otimes C_{{B}}$. Using this, along with the fact that Kronecker products commute with matrix inversion, transposition,  and multiplication, we now see that
\begin{align*}
    \Phi_{{A} \otimes {B}} &= -C_{{A} \otimes {B}}^T C_{{A} \otimes {B}}^{-1} \\
    &= -(C_{A} \otimes C_{B})^T (C_{{A}} \otimes C_{{B}})^{-1} \\
    &= -(C_{A}^T C_{A}^{-1}) \otimes (C_{B}^T C_{B}^{-1}) \\
    &= -(-\Phi_{A}) \otimes (-\Phi_{B} ) \\
    &= - \Phi_{A} \otimes \Phi_{B},
\end{align*}
as desired.
\end{proof}

The following result, in conjunction with the previous one, allows us to compute the Coxeter polynomial of a tensor product of algebras.

\begin{lem} \label{lem:characteristic polynomial of kronecker product}
    Let ${M}$ and ${N}$ be square matrices over a field, and denote their characteristic polynomials by $p_{M}(x)$ and $p_{N}(x)$. Suppose that $\lambda_1\ldots, \lambda_m$ are all the roots of $p_{M}(x)$, counted with multiplicity, and similarly that $\mu_1, \ldots, \mu_n$ are the roots of $p_{N}(x)$. Then the characteristic polynomial of the Kronecker product ${M} \otimes {N}$ is
    $$
    p_{{M} \otimes {N}}(x) = \prod_{i, j} (x - \lambda_i \mu_j).
    $$
\end{lem}
\begin{proof}
To prove the result, it is enough to show that the eigenvalues of ${M} \otimes {N}$ are all of the form $\lambda_i \mu_j$, and that the algebraic multiplicity of the eigenvalue $\lambda_i \mu_j$ is the product of the multiplicities of $\lambda_i$ and $\mu_j$.  To see that this is true, it suffices to check the case where ${M}$ and ${N}$ are both in Jordan normal form with only one Jordan block each, and with eigenvalue $\lambda$ and $\mu$, respectively. Then ${M} \otimes {N}$ is an upper triangular matrix whose diagonal entries are all equal to $\lambda \mu$. The result follows.
\end{proof}

Using the two preceding results, we now compute the Coxeter polynomial of the tensor product of $K\mathbb{A}_n$ with itself.

\begin{prop}
The Coxeter polynomial of $K\mathbb{A}_n \otimes K \mathbb{A}_n$ is
$$
( 1 + x)(1 - (-x)^{n + 1})^{n - 1}.
$$
\end{prop}
\begin{proof}
We note that the Coxeter polynomial of $K\mathbb A_n$ is
\begin{align*}
\phi_{K\mathbb A_n}(x) &= 1 + x + \ldots + x^n \\
&= \frac{x^{n + 1} - 1}{x - 1} \\
&= \prod_{u \in U_{n + 1}\backslash\{1\}} (x - u),
\end{align*}
where $U_{n + 1}$ denotes the group of $(n + 1)$th roots of unity; see e.g.\ page 5 of \cite{de-la-Pena-Lenzing'08}. Then by \cref{lem:characteristic polynomial of kronecker product}, the characteristic polynomial of the Kronecker product $\Phi_{K\mathbb A_n} \otimes \Phi_{K \mathbb A_n}$ is
$$
p_{\Phi_{K\mathbb A_n} \otimes \Phi_{K \mathbb A_n}}(x) = \prod_{u, v \in U_{n + 1}\backslash\{1\} } (x - uv).
$$

In order to determine the multiplicities of the roots of this polynomial, we consider the multiplication map
$$
U_{n + 1}\backslash \{1\} \times U_{n + 1}\backslash \{1\} \to U_{n + 1}.
$$
For a root of unity $u \in U_{n + 1}$, the multiplicity of $u$ as a root of $p_{\Phi_{K\mathbb A_n} \otimes \Phi_{K \mathbb A_n}}(x)$ is equal to the number of pre-images it has under this map. The pre-images of $u$ are all the elements of the form $(v, uv^{-1})$ with $v \in U_{n + 1}$, $v \ne 1$, and $v \ne u^{-1}$. Thus the number of pre-images is $n - 1$ if $u \ne 1$, or $n$ if $u = 1$. Hence, the characteristic polynomial of $\Phi_{K\mathbb A_n} \otimes \Phi_{K \mathbb A_n}$ is
\begin{align*}
    p_{\Phi_{K\mathbb A_n} \otimes \Phi_{K \mathbb A_n}}(x) &= (x - 1)^n \prod_{u\in U_{n + 1}\backslash \{1\}} (x - u)^{n - 1} \\
    &= (x - 1) \prod_{u \in U_{n  + 1}} (x - u)^{n - 1} \\
    &= (x - 1) (x^{n + 1} - 1)^{n - 1}.
\end{align*}

Now we recall from \cref{lem:coxeter matrix of tensor product}
 that the Coxeter matrix of $K \mathbb{A}_n \otimes K \mathbb{A}_n$ is $- \Phi_{K \mathbb{A}_n} \otimes \Phi_{K \mathbb{A}_n}$, and hence
 \begin{align*}
     \phi_{K \mathbb{A}_n}(x) &= \det(xI + \Phi_{K \mathbb{A}_n} \otimes \Phi_{K \mathbb{A}_n}) \\
     &= (-1)^{n^2}\det((-x)I - \Phi_{K \mathbb{A}_n} \otimes \Phi_{K \mathbb{A}_n}) \\
     &= (-1)^{n^2} p_{\Phi_{K \mathbb{A}_n} \otimes \Phi_{K \mathbb{A}_n}}(-x) \\
     &= (-1)^{n^2} (-x - 1)((-x)^{n + 1} - 1)^{n - 1} \\
     &= (-1)^{n^2 + n} (1 + x)(1 - (-x)^{n + 1})^{n - 1} \\
     &= (1 + x)(1 - (-x)^{n + 1})^{n - 1},
 \end{align*}
 as desired.
\end{proof}

\subsection{Graded Cartan determinants and Coxeter polynomials}
For the remainder of this section, we let $A$ be a basic finite dimensional algebra of finite global dimension, 
and we consider $\Delta(A)$ to have the trivial extension grading and $\Pi_{d+1}(A)$ to have the (higher) preprojective grading, i.e.\ the gradings induced by considering them as quotients of tensor algebras over $A$ with $A$ in degree $0$ and the generating bimodule in degree $1$.
Hence, $\Delta(A) = T_{A}D(A)/\langle D(A)^{\otimes_A 2}\rangle$ is graded by putting $A$ in degree $0$ and $D(A)$ in degree $1$. 
Similarly, $\Pi_{d+1}(A) = T_{A}(\tau_d^{-1}A))$ is graded by putting $A$ in degree $0$ and $\tau_d^{-1}A \cong \Ext^d_A(D(A), A)$ in degree $1$. 

\begin{prop}\label{prop: det of graded Cartan matrix of trivial extension}
Let $A$ be a basic finite dimensional $K$-algebra of finite global dimension. Then
$$\det C_{\Delta(A)}(x) = \phi_A(x) \det(C_{A}).$$ 
\end{prop} 
\begin{proof}
It is straightforward to check that $C_{\Delta(A)}(x) = C_{A} + xC_{A}^T$. 
Let $s$ be the number of vertices in the quiver of $A$.
Hence, we compute
\begin{align*}
\det(C_{\Delta(A)}(x))
& = 
\det(C_{A} + xC_{A}^T)\\
& = 
\det(I + xC_{A}^TC_{A}^{-1})\det(C_{A})\\
& = 
x^{s}\det(x^{-1}I + C_{A}^TC_{A}^{-1})\det(C_{A})\\
& = 
\phi_{A}(x)\det(C_{A}),
\end{align*}
where in the last step we use that the Coxeter polynomial of an algebra is self-reciprocal, i.e. satisfies that $\phi_{A}(x) = x^{s}\phi_{A}(x^{-1})$.
\end{proof}
We note that it is well known that $\det(C_A) = \pm 1$ holds if $\gldim A < \infty$. 
Moreover, it is also well known that $\det(C_A) = 1$ holds if $A$ is triangular.

Recall that we let $\lvert A \rvert$ denote the number of simple modules of $A$.
Let $A$ now be a basic $d$-representation finite algebra with presentation $A = KQ/I$ for a quiver $Q$ and an admissible ideal $I$.
Hence, for each indecomposable projective $P_i := e_i A$, there is an $\ell_i$ such that $\nu_d^{-\ell_i + 1}(P_i)$ is indecomposable injective, and thus $\nu_d^{-\ell_i}(P_i) \cong P_{\pi(i)}[d]$ for some permutation $\pi$ defined on the vertices of $A$. 
We can then let $P$ be the matrix corresponding to the induced action on $$\K_0(\D^b(\modu A))$$ by the permutation $\pi$.
We now relabel the vertices $Q_0$ of $A = KQ/I$ such that $Q_0 = \{1, \ldots, \lvert A\rvert\}$ holds, and let $E_i$ be the matrix with all entries being $0$ except for the entry in row $i$ and column $i$, which equals $1$. 
In other words, $(E_i)_{jk} = \delta_{ij}\delta_{ik}$ where $\delta_{jk}$ is the Kronecker delta. 

\begin{prop}\label{prop: product of graded cartan matrices formula}
Let $A$ be $d$-representation finite, and assume $A = KQ/I$ for a quiver $Q$ and an admissible ideal $I$.
Given the setup in the paragraph above, 
$$\det(C_{\Pi_{d+1}(A)}(x)C_{\Delta(A)}((-1)^{d+1}x)) = \det(I - (-1)^{d} \sum_{i \in Q_0}x^{\ell_i}P E_i).$$ 

If $A$ is also $\ell$-homogeneous,  
then 
$$\det(C_{\Pi_{d+1}(A)}(x)C_{\Delta(A)}((-1)^{d+1}x)) = \det(I - ((-1)^{d-1} x\Phi^{-1}_{A})^{\ell}).$$ 
\end{prop}
\begin{proof}
Recall that $\Pi_{d+1}(A)$ is the tensor algebra $T_{A} (\tau_d^{-1} A)$, that $(\tau_d^{-1} A)^{\otimes_{A}^\ell} = 0$ for $\ell$ such that $\ell \geq \ell_i$ for any $i \in Q_0$, and that $(-1)^{d}C^{T}C^{-1}$ is the matrix acting on the Grothendieck group $\K_0(\D^b(\modu A))$ by left multiplication corresponding to $\nu_{d}^{-1}(-)$, where $\nu_d(-) := (- \otimes_{A}^{\mathbb{L}} DA)  \circ [-d]$.
Note that we have assumed the basis of $\K_0(\D^b(\modu A))$ is given by the simple modules, and that we are using column vectors. 
Now let
\[
M := (-1)^{d}xC_A C_A^{-T} = (-1)^{d} x(C_A^{T} C_A^{-1})^{-1}.
\]
Using this, one can check that 
\begin{align*}
C_{\Pi_{d+1}(A)}(x) 
& = \sum_{i \in Q_0} C_{A}E_i + M C_{A}E_i + \ldots + M^{\ell_i - 1}C_{A}E_i \\
& = \sum_{i \in Q_0} \sum_{j = 0}^{\ell_i - 1}M^{j}C_{A}E_i \\
\end{align*}

At this point, by using that 
\begin{align*}
C_{\Delta(A)}((-1)^{d+1}x) 
& = C_A - (-1)^{d}xC_A^T\\
& = C_{A}(I - (-1)^{d}xC_{A}^{-1}C_{A}^T)
\end{align*}
the claim follows by a few more computations. 
Hence, note that
\begin{align*}
\det(C_{\Pi_{d+1}(A)}(x)C_{\Delta(A)}((-1)^{d+1}x)) 
& =
\det(C_{\Pi_{d+1}(A)}(x)) \det(C_{\Delta(A)}((-1)^{d+1}x))\\
& = 
\det(C_{\Pi_{d+1}(A)}(x)) \det(C_{\Delta(A)}((-1)^{d+1}x))^T\\
& = 
\det(\sum_{i \in Q_0} \sum_{j = 0}^{\ell_i - 1} M^{j}C_A E_i)\det(I - M)\det(C_A)\\
& = 
\det(I - M)\det(\sum_{i \in Q_0} \sum_{j = 0}^{\ell_i - 1} M^{j}C_A E_i)\det(C_A)\\
& = 
\det(\sum_{i \in Q_0} C_A E_i - M^{\ell_i}C_A E_i)\det(C_A)\\
& = 
\det(\sum_{i \in Q_0} C_A E_i - ((-1)^{d}xC_A C^{-T})^{\ell_i}C_A E_i)\det(C_A)\\
& = 
\det(\sum_{i \in Q_0} C_A E_i - (-1)^{d}x^{\ell_i}C_A P E_i)\det(C_A)\\
& = 
\det(\sum_{i \in Q_0} E_i - (-1)^{d}x^{\ell_i}P E_i)\det(C_A)^2\\
& = 
\det(I - (-1)^{d}\sum_{i \in Q_0} x^{\ell_i}P E_i).\\
\end{align*}
For the penultimate equality, we use that $\gldim A < \infty$ implies that $\det C_A = \pm 1$.

In case $A$ is $\ell$-homogeneous, we obtain that 
\begin{align*}
\det(I - (-1)^{d}\sum_{i \in Q_0} x^{\ell_i}P E_i)
& = 
\det(I - (-1)^{d}\sum_{i \in Q_0} x^{\ell}P E_i)\\
& = 
\det(I -  (-1)^{d}x^{\ell}P \sum_{i \in Q_0}E_i)\\
& = 
\det(I -  (-1)^{d}x^{\ell}P I)\\
& = 
\det(I -  (-1)^{d}x^{\ell}(-1)^{d(\ell - 1)}(C_A C^{-T})^{\ell})\\
& = \det(I - ((-1)^{d}xC_A C^{-T})^{\ell})\\
& = \det(I - ((-1)^{d-1}x\Phi_{A}^{-1})^{\ell}).
\end{align*}

Here, we use that $\Phi_{A}^{-1} = -C_A C_A^{-T}$ to obtain the last equality.
\end{proof}

The next result is sometimes useful in combination with the preceding proposition. 

\begin{prop}
Let $A$ be an $\ell$-homogeneous $d$-representation finite algebra. Then $$(-1)^{d(\ell - 1) - \ell}\Phi_{A}^{\ell} = (-1)^{d(\ell-1)}(C_A^T C^{-1}_A)^{\ell}$$ is the permutation matrix corresponding to the Nakayama permutation of $\Pi_{d+1}(A)$.
\end{prop}
\begin{proof}
This follows by Proposition 2.4 of \cite{HI11b}.
\end{proof}

In \cite{de-la-Pena-Lenzing'08}, one can find a table listing the Coxeter polynomials of representation finite hereditary algebras. 
Note that the Coxeter polynomial is a derived invariant and every orientiation of a given ADE Dynkin diagram results in derived equivalent algebras.

Let $A$ be hereditary of type $\mathbb{A}_n$. 
Then $\phi_{A}(x) = (1 - x^{n+1})/(1 - x)$ by \cite{de-la-Pena-Lenzing'08}. 
If $n$ is odd and the quiver of $A$ is oriented symmetrically as in Proposition 3.2 of \cite{HI11b}, then $A$ is $\ell$-homogeneous with $\ell = (n + 1)/2$ and $-\Phi_A^{\ell}$ is the $n \times n$ matrix with ones along the antidiagonal and zeros elsewhere. 
This can be confirmed by considering the AR-quiver of $A$. 
To use \cref{thm: etingof-ginzburg}, we thus need to compute $$\det(I - ((-1)^{d-1} x\Phi_{A})^{-\ell}) = \det(I - x\Phi_{A}^{-\ell}).$$ 

\begin{prop}
Let $A$ be hereditary of type $\mathbb{A}_n$ for odd $n$ and with quiver oriented symmetrically.
Then 
$$\det(I - x^{\ell}\Phi_{A}^{-\ell}) = (1 + x^{\ell})(1 - x^{2\ell})^{\ell - 1}.$$
\end{prop}
\begin{proof}
Let $J_m$ be the $m \times m$ matrix with ones along the antidiagonal and zeros elsewhere. 
Then 
\begin{align*}
\det(I - x^{\ell}\Phi_{A}^{-\ell}) 
& =
\det(I - x^{\ell}(-J_n)) \\
& = 
x^{n\ell}\det(x^{-\ell}I + J_n) \\
& = 
x^{n\ell}(-p_{J_n} (-x^{-\ell})),
\end{align*}
where $p_{J_n}(x)$ is the characteristic polynomial of $J_n$. Note that we have used that $n$ is odd.

Since $J_n$ is a permutation matrix, we get that 
\[
-p_{J_n}(x) = (1 - x)(x^2 - 1)^{(n-1)/2} = (1 - x)(x^2 - 1)^{\ell - 1}
\]
and thus that 
\[
-p_{J_n}(-x) = (1+x)(x^2 - 1)^{(n-1)/2} = (1 + x)(x^2 - 1)^{\ell - 1}
\]
Hence, 
\begin{align*}
\det(I - x^{\ell}\Phi_{A}^{-\ell}) 
& = x^{n\ell}(1 + x^{-\ell})(x^{-2\ell} - 1)^{(n-1)/2}\\
& = x^{\ell}(1 + x^{-\ell})x^{(n-1)\ell}(x^{-2\ell} - 1)^{(n-1)/2}\\
& = (x^{\ell} + 1)(1 - x^{2\ell})^{(n-1)/2}\\
& = (1 + x^{\ell})(1 - x^{2\ell})^{\ell - 1},
\end{align*}
which is what was to be shown.
\end{proof}

In \cref{iex: Etingof-Eu}, we omitted showing how one computes the graded Euler characteristic. 
We now show how this can be done using our results. 
For the sake of simplicity, we demonstrate only in Dynkin type $\mathbb{A}_n$ with odd $n$ and with quiver oriented symmetrically.

\begin{example}
Let $A$ be of hereditary of type $\mathbb{A}_n$ for odd $n$ and with symmetric orientiation. 
Also let $\Pi := \Pi(A)$.
We now want to compute $\chi_{\overline{\HC}_*(\Pi)}(x) = \sum_{k \geq 1} a_k x^k$ by using \cref{thm: etingof-ginzburg}:

\begin{align*}
\prod\limits_{k=1}^\infty(1-x^k)^{-a_k}
& =
\prod\limits_{s=1}^\infty \det C_\Pi(x^s) \\
& =
\prod\limits_{s=1}^\infty \det\left(\frac{I - x^{s\ell}\Phi_A^{-\ell}}{I + x^s\Phi_{A}^{-1}}\right) \\
& = 
\prod\limits_{s=1}^\infty \frac{(1 + x^{s\ell})(1 - x^{s2\ell})^{\ell-1}}{\frac{1 - x^{s2\ell}}{1 - x^s}} \\
& = 
\prod\limits_{s=1}^\infty 
\frac{(1-x^s)(1 - x^{s2\ell})^{\ell-1}}{1 - x^{s\ell}}. \\
\end{align*}
Hence,  we get
\[
\chi_{\overline{\HC}_*(\Pi)}(x) = \frac{x^\ell}{1 - x^\ell} -\frac{x}{1 - x}  - (\ell - 1)\frac{x^{2\ell}}{1 - x^{2\ell}}.
\]
To compare with the result in \cite{Etingof-Eu'06}, we compute 
\begin{align*}
\chi_{\overline{\HC}_*(\Pi)}(x) 
& = 
\frac{x^\ell}{1 - x^\ell} -\frac{x}{1 - x}  - (\ell - 1)\frac{x^{2\ell}}{1 - x^{2\ell}}\\
& = 
-\frac{1}{1 - x^{2\ell}}\left(-x^{\ell} - x^{2\ell} + \sum_{i = 1}^{2\ell}x^{i} + (\ell - 1)x^{2\ell}\right)\\
& = 
-\frac{1}{1 - x^{2\ell}}\left(\sum_{i = 1}^{2\ell- 1}x^{i} -x^{\ell} + (\ell - 1)x^{2\ell}\right).
\end{align*}
The exponents $\{m_i\}_{1}^{n}$ associated to the root system $\mathbb{A}_n$ are given by $\{1,2,\ldots,n\}$. 
Moreover, we always have that $2\ell - 1 = n$, that $-(r_{+} - r_{-}) = - 1$, and that $\ell - 1 = r_{-}$; see \cite{Etingof-Eu'06} for the definition of the notation $r_{-}$ and $r_{+}$. 
Hence, we see that we obtain the same expression as that in \cite{Etingof-Eu'06} by making the substitution $x \mapsto x^2$. 
\end{example}

We need the following for our application in \cref{sec:computing homology}. 

\begin{prop} \label{prop: det of I - Cox for 2-RF of type A 2n + 1 times A}
Let $A$ be $2$-representation finite of type $\mathbb{A}_{n} \times \mathbb{A}_{n}$ for odd $n$ and where the quiver of each factor has a symmetric orientation.
Then 
$$\det(I - (-x)^{\ell}\Phi_{A}^{-\ell}) = (1 - x^{\ell})(1 - x^{2\ell})^{(n^2 - 1)/2}.$$
\end{prop}
\begin{proof}
As before, we let $J_m$ be the $m \times m$ matrix with ones along the antidiagonal and zeros elsewhere. 
Our assumptions imply that $A$ is $\ell$-homogeneous with $\ell = (n + 1)/2$, and that the Nakayama permutation matrix for $\Pi_3(A)$ is given by the Kronecker product of two copies of the Nakayama permutation matrix of the preprojective of $\mathbb{A}_{n}$ with bipartite orientation. 
In other words, $(-1)^{\ell}\Phi_{A}^{-\ell}$ is equal to $J_{n^2}$.
Then 
\begin{align*}
\det(I - (-x)^{\ell}\Phi_{A}^{-\ell}) 
& =
\det(I - x^{\ell}J_{n^2})) \\
& =
x^{n^2 \ell}\det(x^{-\ell}I - J_{n^2})) \\
& = x^{n^2 \ell}p_{J_{n^2}} (x^{-\ell}),
\end{align*}
where $p_{J_{n^2}}(x)$ is the characteristic polynomial of $J_{n^2}$. 

Since $J_{n^2}$ is a permutation matrix, we get that 
\[
p_{J_{n^2}}(x) = (x-1)(x^2 - 1)^{(n^2-1)/2}
\]
and thus that 
\[
p_{J_{n^2}}(x^{-\ell}) = (x^{-\ell} - 1)(x^{-2\ell} - 1)^{(n^2-1)/2}.
\]
Hence, 
\begin{align*}
\det(I - (-x)^{\ell}\Phi_{A}^{-\ell}) 
& = x^{n^2\ell}(x^{-\ell} - 1)(x^{-2\ell} - 1)^{(n^2 - 1)/2}\\
& = x^{\ell}(x^{-\ell} - 1)x^{(n^2-1)\ell}(x^{-2\ell} - 1)^{(n^2 - 1)/2}\\
& = (1 - x^{\ell})(1 - x^{2\ell})^{(n^2 - 1)/2},
\end{align*}
which is what was to be shown.
\end{proof}

\section{A computation of $\HC_* \Pi, \HH_* \Pi$ and $\HH^*\Pi$ for $\Pi = \Pi_{2+1}(K\mathbb{A}_n^{\otimes 2})$ for odd $n$ and $\mathbb{A}_n$ oriented symmetrically} \label{sec:computing homology}
In this section, we compute the cyclic homology, and the Hochschild homology and cohomology for the higher preprojectives of the $2$-hereditary algebras obtained as tensor products of representation finite hereditary algebras of type $\mathbb{A}_n$.
To do this, we recall a description of the higher preprojective algebra of a tensor product of two $d$-representation finite algebras as a Segre product. 
Using this, we compute $\HH_0$ and $\HH^0$ of such a higher preprojective algebra $\Pi$. 
Following this, we then compute the graded Euler characteristic of such a $\Pi$ by using the Euler characteristic formula from \cref{sec:product formula for the euler characteristic} and the results of \cref{sec: some linear algebra}.
Finally, we apply this graded Euler characteristic to the Connes exact sequence from \cref{sec:reduced homology} in combination with some duality formulas from \cref{subsection: CY Frobenius and HH^*} and compute $\HC_* \Pi, \HH_*\Pi$ and $\HH^* \Pi$.

\subsection{Segre products and higher preprojective algebras}
Roughly speaking, the next result says that the higher preprojective algebra of a $d$-representation finite algebra that arises as a tensor product is the Segre product of the higher preprojective algebras of the tensor factor algebras. 
Hence, we recall that if $\Lambda$ and $\Lambda'$ are two positively graded algebras, their \textit{Segre product} --- denoted $\Lambda \overline \otimes \Lambda'$ --- is the positively graded algebra given by $\oplus_{i \geq 0}\Lambda_i \otimes \Lambda'_i$.
Note that we assume the base field $K$ to be perfect for the remainder of this subsection.

\begin{prop}[See {\cite[Corollary 4.5]{Thi20}}]\label{prop: preproj of tensor product is Segre product of preproj} 
Let $A_1$ and $A_2$ be, respectively, $d_1$- and $d_2$-representation finite, set $d := d_1 + d_2$, and consider all (higher) preprojective algebras to be endowed with the relevant (higher) preprojective gradings.
Then $$\Pi_{d+1}(A_1 \otimes_K A_2)_i \cong \Pi_{d_1+1}(A_1) \overline \otimes \Pi_{d_2 + 1}(A_2).$$
\end{prop}
\begin{proof}
The proof is essentially the same as that of Corollary 4.5 of \cite{Thi20}, i.e.\ the statement follows by the description of a $(d+1)$-preprojective algebra $\Pi_{d+1}(A)$ as a tensor algebra over $A$ and the description of $\tau_d^{-1}(A_1 \otimes_K A_2)$ as the tensor product $\tau_{d_1}^{-1}(A_1) \otimes_K \tau_{d_2}^{-1}(A_2)$.
\end{proof}

Let $\Lambda = \oplus_{i\geq 0}\Lambda_i$ be a positively graded algebra. 
Recall that $Z\Lambda$ denotes the center of $\Lambda$, and let 
\[
Z_{\Lambda_0}(\Lambda) := \{ z \in \Lambda \, \lvert \,  z \cdot \lambda = \lambda \cdot z \, \, \mathrm{for} \, \mathrm{all} \,\, \lambda \in \Lambda_0 \}.
\]
Note that if $\Lambda = KQ/I$ is a finite dimensional algebra and the idempotents $e_i$ for $i \in Q_0$ are in $\Lambda_0$, then any $z \in Z_{\Lambda_0}(\Lambda)$ is a linear combination of cycles since $e_i z = z e_i$. 

\begin{prop}\label{prop: center of preproj of tensor product}
Let $\Lambda$ and $\Gamma$ be graded algebras. Then
$$Z(\Lambda \overline{\otimes} \Gamma) \subseteq Z_{\Lambda_0}(\Lambda) \overline{\otimes} Z_{\Gamma_0}(\Gamma).$$
\end{prop} 
\begin{proof} 
Let $z \in Z(\Lambda \overline{\otimes} \Gamma)$, and write $z = \sum_j z_{1,j} \otimes z_{2,j}$. Note that we can assume without loss of generality that $\{z_{2,j}\}$ is linearly independent over $K$. Now let $\lambda \in \Lambda_0$ be a homogeneous element of degree zero, and note that
$$
(\lambda \otimes 1) \cdot \sum_{j} z_{1, j} \otimes z_{2, j} = \sum \lambda z_{1, j} \otimes z_{2, j}
$$
and that
$$
\sum_j z_{1,j} \otimes z_{2,j} \cdot (\lambda \otimes 1)  = \sum_j z_{1,j} \lambda  \otimes z_{2,j},
$$
so that
$$
\sum_{j} (\lambda z_{1, j} - z_{1, j} \lambda)\otimes z_{2, j} = 0.
$$
We now deduce that $\lambda z_{1, j} = z_{1,j}\lambda$ since the $\{ z_{2, j}\}$ are assumed to be linearly independent, and hence $z_{1, j} \in Z_{\Lambda_0}(\Lambda)$. Similarly, one may choose the $\{z_{2, j}\}$ to be elements of $Z_{\Gamma_0}(\Gamma)$.
\end{proof}

\subsection{Preprojective algebras of type $\mathbb{A}$}\label{subsection: Preprojective algebras of type A}
For the remainder of the paper, we assume that the quiver of the preprojective algebra of type $\mathbb{A}_n$ is given as follows. 
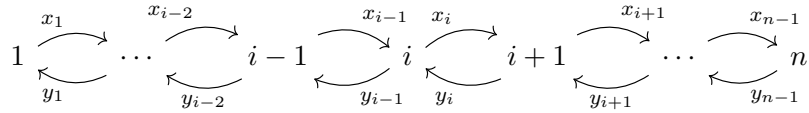
\[
\begin{tikzcd}
1 \rar[bend left, "x_1"] & \cdots \lar[bend left, "y_1"] \rar[bend left, "x_{i - 2}"] 
& i - 1 \lar[bend left, "y_{i - 2}"] \rar[bend left, "x_{i - 1}"] & i \lar[bend left, "y_{i - 1}"] \rar[bend left, "x_{i}"]
& i + 1 \lar[bend left, "y_{i}"] \rar[bend left, "x_{i+1}"] & \cdots \lar[bend left, "y_{i+1}"] \rar[bend left, "x_{n - 1}"]
& n \lar[bend left, "y_{n-1}"]
\end{tikzcd}
\]

\begin{remark}\label{rem: signs in relations of classical preproj of type A}
While the quiver of the preprojective of a hereditary algebra $A$ is the same for the various different orientations of the underlying graph of the quiver of $A$, the same may not be true for the relations.
However, in the case of type $\mathbb{A}_n$, one can use base change arguments such as in \cite{Erdmann-Solberg-1} to show that the relations can be assumed to be of the following form: $x_1y_1, x_2y_2 - y_1 x_1, \ldots, x_i y_i - y_{i-1}x_{i-1}, x_{n-1}y_{n-1} - y_{n-2}x_{n-2}, y_{n-1}x_{n-1}$.
\end{remark}

The following is well known, but we include a proof for the sake of the reader.

\begin{prop}\label{prop: preproj and graded Frob}
Let $\Pi$ be the preprojective algebra of type $\mathbb{A}_n$. 
Then the following hold. 
\begin{enumerate}[(i)]
    \item $\Pi$ is a graded Frobenius algebra of highest degree $a = n - 1$ when endowed with the grading given by putting each arrow in degree $1$.
\item $\Pi$ has Nakayama permutation given by $i \mapsto n + 1 - i$.
\end{enumerate}
\end{prop}
\begin{proof}
For the first claim, note that it is well-known that the highest degree of $\Pi$ is $n-1$ with this grading, but that this can also be checked by observing that $x_1 x_2 \cdots x_{n-1}$ lies in the socle of $e_1 \Pi$ and using \cite[Theorem 3.3]{Martinez-Villa'99} to deduce that all the projectives of $\Pi$ have the same Loewy length. 
Then the claim follows by e.g.\ \cite[Lemma 2.4]{HS}.

The second claim then follows by e.g.\ Proposition 2.2 of \cite{HI11b} and their explication on pages 454-455 of the relevant involutions defined on Dynkin quivers.
\end{proof}

\subsection{$\HH^0$ and $\HH_0$ of higher preprojective algebras of $K\mathbb{A}_{n}^{\otimes 2}$ for $n$ odd and $\mathbb{A}_{n}$ oriented symmetrically}
By \cite{HI11b} we know that for a tensor product of algebras of type $\mathbb{A}_n$ to be $2$-representation finite, each tensor factor must be $\ell$-homogeneous for the same $\ell$.
Hence, the parameter $n$ must be the same for each tensor factor, and each tensor factor must be \textit{oriented symmetrically}: i.e.\ meaning that if the quiver of type $\mathbb{A}_n$ is a subquiver of the quiver given in \cref{subsection: Preprojective algebras of type A}, then the involution $i \mapsto n + 1 - i$ does not change the orientation of the subquiver. 
Consequently, one also sees that $n$ must be odd if a quiver of type $\mathbb{A}_n$ is to have a symmetric orientation. 

Thus, we now need to recall and/or establish some technical results about the preprojective algebra of type $\mathbb{A}_n$ for odd $n$.  

\begin{prop}\label{prop: technical preproj result}
Let $\Pi$ be the preprojective algebra of type $\mathbb{A}_n$ for odd $n$, and assume that it is endowed with the grading given by putting each arrow in degree $1$.
Then the following hold. 
\begin{enumerate}[(i)]
    \item Any cycle at a vertex $i$ is of degree $2m$ for some natural number $m$, and is, modulo the relations, equal to the cycle $\sigma_{i}^{m} := (x_i y_i)^m$ up to a scalar.
    \item For $i \leq (n-1)/2+1$, we have $\sigma_{i}^{m} \neq 0$ if and only if $0 \leq m \leq i - 1$.
     \item For $i \geq (n-1)/2+1$, we have $\sigma_{i}^{m} \neq 0$ if and only if $0 \leq m \leq n - i$. 
\end{enumerate}
\end{prop}
\begin{proof}
By the form of the quiver, it is clear that any cycle must be of even degree; and, similarly, by also considering the form of the relations, it is clear that the second part of \textit{(i)} follows as well. 

We now show \textit{(ii)}, and for this we proceed by cases. 
Now, it is straightforward to check that $\sigma_i^{m} = 0$ for $m \geq i$ as $\sigma_{i}^{i - 1}$ can be written as $y_{i - 1} y_{i - 2} \cdots y_1 x_1 \cdots x_{i - 2} x_{i - 1}$ so that if the exponent were greater, the result would be zero as we could rewrite to include the relation $x_1 y_1$.

We now argue that $\sigma_{i}^{i-1} \neq 0$. 
To begin with, note that by our assumption that $n$ is odd, we can apply \cref{prop: preproj and graded Frob} to $\Pi$. 
With this in mind, we observe that if $j = (n-1)/2 + 1$, then the Nakayama permutation maps $j$ to itself. 
Consequently, the socle of $e_j \Pi$ must include a (non-zero) cycle at $j$ since $\socM (e_j \Pi)e_k$ is non-trivial if and only if $k$ is the image of $j$ under the Nakayama permutation due to $\Pi$ being a Frobenius algebra; see e.g.\ \cite{Skowronski}. 
Moreover, this cycle must be of degree $n - 1$ by \cref{prop: preproj and graded Frob}.
This must then equal $\sigma_j^{j - 1}$ up to a scalar by part \textit{(i)} above.
It now suffices to observe that for any $i \leq j = (n-1)/2 + 1$ we have $\sigma
^{j-1}_j = y_{j-1} \cdots y_i \sigma_i^{i - 1} x_i \cdots x_{j-1}$. 

The case of $i \geq (n-1)/2 + 1$ is similar and its proof is thus omitted. 
Hence, we are done.  
\end{proof}

We are now ready to prove the following.

\begin{prop}\label{prop: center relative to degree 0 of preproj of type A_n}
Let $\Pi$ be the preprojective of $K\mathbb{A}_n$ with $n$ odd. Then $Z_{K\mathbb{A}_n}(\Pi) = Z(\Pi)$.
\end{prop}
\begin{proof}
Note that we may assume that we are dealing with a homogeneous element $z \in Z_{k\mathbb{A}_n}(\Pi)_{2m}$ with respect to the grading putting all arrows in degree $1$ since any grading on an algebra induces a grading on the center of that algebra. 
We proceed by establishing the following four properties of such an element $z$.
\begin{enumerate}[(i)]
    \item $z = \sum_{i = j}^{J}C_i \sigma_i^m$ for some positive integers $j,J$ and $m$ satisfying that $C_j\sigma_j^m, C_J\sigma_J^m \neq 0$.
    \item $m = j - 1$ for $m$ and $j$ as in \textit{(i)}.
    \item $J = n + 1 - j$ for $j$ and $J$ as in \textit{(i)}.
    \item $C_i = C_{i'}$ for $i,i'$ satisfying $j \leq i, i'$ and $i, i' \leq n+1 - j$ for $j$ as in \textit{(i)}. 
\end{enumerate}
Indeed, observe that \textit{(i)} through \textit{(iv)} imply that $z \in Z(\Pi)$ holds by a straightforward computation. 
Namely, one uses that for $\Lambda$ a quotient of a path algebra $KQ$ over a quiver $Q$ one has $z \in Z(\Lambda)$ if and only if $e_i z = z e_i$ and $\alpha z = z \alpha$ for all $i \in Q_0$ and $\alpha \in Q_1$.
Since we can set $C_i = 1$ for all $j \leq i \leq J$, the claim follows by using this and observing that 
\[
\sigma_k^m x_k = (x_ky_k)^mx_k = x_k(y_kx_k)^m = x_k \sigma_{k+1}^m,
\]
that
\[
y_k\sigma_k^m  = y_k(x_ky_k)^m = (y_kx_k)^my_k =  \sigma_{k+1}^m y_k,
\]
and that all of the terms occurring in the above are non-zero if and only if $j = m + 1 \leq k \leq n -  m = J$.
Consequently, it is sufficient to establish that \textit{(i)} through \textit{(iv)} all hold.

For \textit{(i)}, note that $z$ must be a sum of oriented cycles as $e_i z = z e_i$ must hold for every vertex $i$ in the quiver of $\Pi$. 
Hence, we can certainly write $z = \sum_{i = 1}^{n} C_i \sigma_{i}^{m}$ and pick $j$ and $J$ to be, respectively, the smallest and largest indices such that $C_j\sigma_j^m, C_J\sigma_J^m \neq 0$ hold.

For \textit{(ii)} through \textit{(iv)}, we note that they all hold if $z \in \socM \Pi$, i.e.\ if $z$ is equal to a scalar multiple of $\sigma_{(n-1)/2 + 1}^{(n-1)/2}$. 
Indeed, in this case, we obtain that $j  = (n-1)/2 + 1 = J$ and $m = (n-1)/2$, showing that \textit{(ii)} through \textit{(iv)} all hold.

In other words, it remains to check \textit{(ii)} through \textit{(iv)} in the case of $z \not \in \socM \Pi$, or, equivalently, in the case of $z x_{k} \neq 0$  for some $1 \leq k \leq n$.
For this, we begin by showing \textit{(ii)}.
Hence, observe that one now has $0 \neq z x_{k} = C_{k} \sigma_{k}^m x_{k}$, which of course implies that $C_{k} \neq 0$ and that $\sigma_{k}^m \neq 0$.
Observe that if $m < k - 1$, then $\sigma_k^{m+1} \neq 0$ must hold by \cref{prop: technical preproj result}, implying 
\[
0 \neq \sigma_{k}^{m+1} = y_{k-1}x_{k-1}\sigma_k^m = \sigma_k^m y_{k-1}x_{k-1}
\]
so that we also know that $x_{k-1}\sigma_k^m \neq 0$ and $\sigma_k^m y_{k-1} \neq 0$ hold. 

Now, while exactly one of $x_{k-1}$ and $y_{k-1}$ is in degree $0$, in either case, we can deduce that $C_{k-1} = C_{k} \neq 0$ when $m < k - 1$. 
Note that we may assume the relations of $\Pi$ to be as described in \cref{rem: signs in relations of classical preproj of type A}.
Thus, by repeating the argument if need be, we see that we could just as well have chosen $k = m + 1$ to begin with. 
Consequently, we have thus shown \textit{(ii)} with $j = k = m + 1$ as $\sigma_{i}^{m} = 0$ holds for $i < m$ by \cref{prop: technical preproj result}. 

We continue by showing \textit{(iii)} and \textit{(iv)} more or less at the same time.
Observe then that what we showed in the preceding paragraph implies that $j = m + 1 < (n-1)/2 + 1$ by our assumption that we are in the case of $z \not \in \socM \Pi$. 
Hence, using this and that
\[m \leq {j} - 1 = ({j} + 1) - 2 = \cdots = ((n-1)/2 + 1) - ((n-1)/2 + 1) - ({j}-1)),\]
we note that \cref{prop: technical preproj result} implies that $\sigma_{j+i}^{m+i} \neq 0$ for $0 \leq i \leq (n-1)/2 + 1 - j$ as well. 

Observe now that
\begin{align*}
\sigma_{j+1}^{m+1} 
& = (y_j x_j)^{m+1} \\
& = y_j (x_j y_j)^m x_j \\
& = y_j \sigma_{j}^m x_j \\
\end{align*}
thus implying that $\sigma_{j}^mx_j, y_j \sigma_{j}^m \neq 0$ if $\sigma_{j+1}^{m+1} \neq 0$. 
In fact, $\sigma_{j+1}^{m+1} \neq 0$ must hold since $z \not \in \socM \Pi$ by assumption and by applying \cref{prop: technical preproj result} \textit{(ii)} and \textit{(iii)}.
Hence, note that similarly to before, we can also deduce that $C_{j + 1} = C_j \neq 0$ by using that one of $x_{j}$ and $y_{j}$ is in degree $0$.
Moreover, since $\sigma_{j+1}^{m+1} = \sigma_{j+1}^{m} x_{j+1} y_{j+1}$ and $\sigma_{j+1}^{m+1} \neq 0$ if $j + 1 \leq (n-1)/2 + 1$, we get that $\sigma_{j+1}^{m} x_{j+1} \neq 0$ so that 
$$zx_{j+1} = C_{j+1} \sigma_{j+1}^{m} x_{j+1} \neq 0,$$ 
and we can now repeat the argument above with $j+1$ instead of $j$ and obtain that $C_{j+2} = C_{j+1}$ if $j + 1 \leq (n-1)/2 + 1$. 
Continuing in this way, we deduce that $C_i = C_{i'}$ for $j \leq i, i' \leq (n-1)/2 + 1$. 

At this point, note that we evidently also have 
\[j - 1 = n - (n + 1 - j) \]
\[j = n - (n + j) \]
\[\vdots\]
\[(n-1)/2 + 1 = n - (n + (n-1)/2 + 1), \]
so that 
\[
m \leq n - (n + 1 - j) = n - (n + 1 - (j + 1)) - 1 = \cdots = n - ((n-1)/2 + 1) - ((n-1)/2 - j),
\]
thus implying that $\sigma_{(n+1 - j) - i}^{m + i} \neq 0$ for $0 \leq i \leq (n-1)/2 + 1 - j$ by \cref{prop: technical preproj result}.
In particular, we also have that $\sigma_{i}^{m+1} \neq 0$ for $(n-1)/2 + 1 \leq i \leq n + 1 - j$. 
Hence, arguing as above, we deduce that $C_i = C_{i'}$ holds for $j \leq i \leq n+1 - j$, and would be done if we knew that $C_i = 0$ for $i > n+1 -j$.
However, we know this to hold by \cref{prop: technical preproj result} as $n - i < j - 1 = m$ must hold for such $i$, thus implying $\sigma_{i}^{m} = 0$.
Consequently, we deduce that $J = n + 1 - j$, thus showing \textit{(iii)}, which immediately establishes \textit{(iv)} since we have shown that $C_i = C_{i'}$ for $j \leq i \leq n + 1 - j$.

By our remark at the beginning of the proof, we are done. 
\end{proof}

We are now finally ready to compute the center of the higher preprojective algebras of the $2$-representation finite algebras obtained as tensor products of representation finite hereditary algebras of Dynkin type $\mathbb{A}$.

\begin{prop}\label{prop: center of higher preproj of type A_{2m + 1} times A_{2m + 1}}
Let $\Pi = \Pi_3(K\mathbb{A}_{n}^{\otimes 2})$ for odd $n$ and $\mathbb{A}_{n}$ oriented symmetrically, and consider all (higher) preprojective algebras to be endowed with the relevant (higher) preprojective gradings.
Then $$Z (\Pi) = Z(\Pi_2(K\mathbb{A}_{n})) \overline\otimes Z(\Pi_2(K\mathbb{A}_{n})).$$
\end{prop}
\begin{proof}
This follows by \cref{prop: center relative to degree 0 of preproj of type A_n} together with \cref{prop: center of preproj of tensor product}.
\end{proof}

Recall that $\HH_0 \Lambda \cong \Lambda/[\Lambda, \Lambda]$; see e.g.\ \cite{loday}.

\begin{prop}\label{prop: HH_0 of Pi of A_n otimes A_n}
Let $\Pi = \Pi_3(K\mathbb{A}_{n}^{\otimes 2})$ for odd $n$ and $\mathbb{A}_{n}$ oriented symmetrically. 
Then $$\HH_0 \Pi = \Pi/\rad \Pi.$$
\end{prop}
\begin{proof}
This is a straightforward computation using that $\HH_0 \Lambda \cong \Lambda/[\Lambda, \Lambda]$ holds for any algebra $\Lambda$: Indeed, if $\Lambda$ is a quotient of a path algebra, it is clear that any element of $\Lambda/[\Lambda, \Lambda]$ is a sum of oriented cycles. 
We can thus consider a cycle $\sigma := \alpha_1 \cdots \alpha_i \alpha_{i+1}$ in the quiver of $\Pi$ and use that
\[
\alpha_1 \cdots \alpha_i \alpha_{i+1} - \alpha_{i+1} \alpha_1 \cdots \alpha_i
\]
is in $[\Pi, \Pi]$ in combination with the description of $\Pi$ as a Segre product given by \cref{prop: preproj of tensor product is Segre product of preproj} to deduce that any such $\sigma$ is zero modulo $[\Pi, \Pi]$ unless $\sigma$ is a scalar multiple of an idempotent.
\end{proof}

\subsection{Computing $\HC_* \Pi, \HH_* \Pi$ and $\HH^* \Pi$ via the graded Euler characteristic}
We are now almost ready to begin computing the cyclic homology, the Hochschild homology, and the Hochschild cohomology of these algebras. 
The only thing that must be done before this is to compute their graded Euler characteristic. 
Since we want to use the results of \cref{sec:product formula for the euler characteristic}, the later parts of \cref{sec: On Frobenius algebras and d-hereditary algebras}, and the first part of \cref{sec: some linear algebra}, we assume that the base field $K$ is algebraically closed and of characteristic zero for the remainder of this subsection.

\begin{prop}
Let $\Pi = \Pi_3(K\mathbb{A}_{n}^{\otimes 2})$ for odd $n$ with $\mathbb{A}_{n}$ oriented symmetrically. 
Then 
\[
\chi_{\overline{\HC}_*(\Pi)}(x) = \frac{1}{1 - x^{n+1}} \left( \sum_{i = 1}^{n} x^{i} - x^{(n+1)/2} - \frac{(n-1)^2}{2}x^{n+1} \right). 
\]
\end{prop}
\begin{proof}
By \cref{prop: product of graded cartan matrices formula}, \cref{prop: det of graded Cartan matrix of trivial extension}, and \cref{prop: det of I - Cox for 2-RF of type A 2n + 1 times A} we have that 
\[
\det (C_\Pi(x)) = (1 - x^{(n+1)/2})(1 - x^{n+1})^{(n^2 - 1)/2}(1-x)^{-1}(1 - x^{n+1})^{-(n - 1)}.
\]
Hence, combining this and \cref{thm: etingof-ginzburg} yields that
\begin{align*}
\chi_{\overline{\HC}_*(\Pi)}(x) 
& = \frac{x}{ 1 - x} - \frac{x^{(n+1)/2}}{1 - x^{(n+1)/2}} - (n-1)^2/2 \frac{x^{n+1}}{1 - x^{n+1}}\\
& = \frac{1}{1 - x^{n+1}} \left(\sum_{i = 1}^{n+1} x^{i} - (1 + x^{(n+1)/2}) x^{(n+1)/2} - \frac{(n-1)^2}{2}x^{n+1} \right)\\
& = \frac{1}{1 - x^{n+1}} \left(\sum_{i = 1}^{n} x^{i} - x^{(n+1)/2} - \frac{(n-1)^2}{2}x^{n+1} \right). 
\end{align*}
\end{proof}

The diagram below can be derived from the Connes long exact sequence recalled in \cref{sec:reduced homology} if one assumes
\begin{enumerate}[(i)]
    \item that $\Lambda$ is graded Frobenius of highest degree $\ell - 1$; 
    \item that $\Lambda$ is a Calabi--Yau Frobenius algebra of dimension $m = 3$ --- see \cref{def: CY Frobenius} --- satisfying $\Omega^{3+1}_{\Lambda^{e}}\Lambda \cong \Lambda_{\mu}\langle \ell \rangle$ where $\mu$ is the Nakayama automorphism of $\Lambda$; 
    \item that $\Lambda$ satisfies the conclusions of \cref{prop: degrees of generators of terms of bimodule projective resolution} and \cref{prop: graded version of Happel's result}; and
    \item that $\mu^2 = 1$.
\end{enumerate}
We note that all but the last item above are satisfied if $\Lambda = \Pi_{d+1}(A)$ for $A$ an $\ell$-homogeneous $2$-representation finite algebra, whereas the last item is also satisfied if $A$ is obtained as a tensor product of two Dynkin hereditary algebras. 

\[
\begin{tikzcd}
\text{degree} & & &\\
& 0 \dar & \\
0 \leq \deg \leq \ell - 1 & \overline{\HH}_0(\Lambda) \dar["B_0"] \rar[equal, shorten = 2mm, shift right] & C \dar["\sim" {rotate=90, anchor=north}] &\\
1 \leq \deg \leq \ell & \overline{\HH}_1(\Lambda) \dar["B_1"] \rar[equal, shorten = 2mm, shift right] & C \rar[phantom, "\oplus"] & X_1 \dar["\sim" {rotate=90, anchor=north}]\\
1 \leq \deg \leq \ell & \overline{\HH}_2(\Lambda) \dar["B_2"] \rar[equal, shorten = 2mm, shift right] & X_2 \dar["\sim" {rotate=90, anchor=north}] \rar[phantom, "\oplus"] & X_1\\
1 \leq \deg \leq \ell & \overline{\HH}_3(\Lambda) \dar["B_3"] \rar[equal, shorten = 2mm, shift right] & X_2 \rar[phantom, "\oplus"] & X_3 \dar["\sim" {rotate=90, anchor=north}]\\
\ell \leq \deg \leq 2\ell - 1 & \overline{\HH}_4(\Lambda) \dar["B_4"] \rar[equal, shorten = 2mm, shift right]  & D(X_2) \langle 2\ell \rangle \dar["\sim" {rotate=90, anchor=north}] \rar[phantom, "\oplus"] & D(X_3) \langle 2\ell \rangle\\
\ell + 1 \leq \deg \leq 2\ell & \overline{\HH}_5(\Lambda) \dar["B_5"] \rar[equal, shorten = 2mm, shift right]  & D(X_2) \langle 2\ell \rangle \rar[phantom, "\oplus"] & D(X_1) \langle 2\ell \rangle \dar["\sim" {rotate=90, anchor=north}]\\
\ell + 1 \leq \deg \leq 2\ell & \overline{\HH}_6(\Lambda) \dar["B_6"] \rar[equal, shorten = 2mm, shift right]  & D(C) \langle 2\ell \rangle \dar["\sim" {rotate=90, anchor=north}] \rar[phantom, "\oplus"] & D(X_1) \langle 
2\ell \rangle\\
\ell + 1 \leq \deg \leq 2\ell & \overline{\HH}_7(\Lambda) \dar["B_7"] \rar[equal, shorten = 2mm, shift right] & D(C) \langle 2\ell \rangle \rar[phantom, "\oplus"] & L\langle 2\ell \rangle \dar["\sim" {rotate=90, anchor=north}]\\
2\ell \leq \deg \leq 3\ell - 1 & \overline{\HH}_8(\Lambda) \rar[equal, shorten = 2mm, shift right]  & C \langle 2\ell \rangle \rar[phantom, "\oplus"] & L\langle 2\ell \rangle\\
\end{tikzcd}
\]

Assumptions (i) and (iii) together yield the bounds on the degrees in the leftmost column by using \cref{prop: HH_* bounds on degrees}. 
One then uses \cref{prop: HH formula for periodic + CY Frob} together with (iii) and (iv) to deduce that 
\begin{align*}
\HH_{i,j}(\Lambda) 
& \cong D\HH_{7 - i, 2\ell - j}(\Lambda)
\end{align*}
holds for $1 \leq i \leq 6$.
In other words, we obtain that $\HH_{7 - i}(\Lambda) \cong D(\HH_i(\Lambda))\langle 2\ell \rangle$ holds for $1 \leq i \leq 6$. 
We can also observe that, on the one hand, $C \oplus X_1$ being trivial in degree $0$ implies that $D(C)\langle 2 \ell \rangle \oplus D(X_{1})\langle 2 \ell \rangle$ must be trivial in degree $2 \ell$. 
On the other hand, $L\langle 2 \ell \rangle$ must be concentrated in degree $2 \ell$ due to the bounds on the degrees that come from \cref{prop: HH_* bounds on degrees}.
In fact, also note that by using the estimates for $X_i$ and $D(X_{i})\langle 2\ell \rangle$, one can also deduce that the degree $\ell$ part of $X_i$ is trivial for $i = 1,2$. 
Similarly, we see that $X_3$ is concentrated in degree $\ell$.

Assume now that we are in the setup of the preceding proposition, i.e.\ $\Lambda = \Pi = \Pi_{2+1}(K\mathbb{A}_n^{\otimes 2})$ with $n$ an odd positive integer and $\mathbb{A}_n$ oriented symmetrically. 
In this case, we obtain that $X_3$ must be trivial as the coefficient of $x^{\ell}$ is zero since $\ell = (n+1)/2$. 
By comparing with the coefficient of $x^{2\ell} = x^{n+1}$, we similarly obtain that $L$ is of dimension $(n-1)^2/2$.

To get more, we have to know more about the reduced Hochschild homology. 
To begin with, we note that we know that $C$ is trivial since $\HH_0(\Pi) = \Pi/\rad \Pi$ by \cref{prop: HH_0 of Pi of A_n otimes A_n}.
Following this, we want to use that the zeroth stable Hochschild cohomology $\stHH^0(\Pi)$ is dual to $\stHH_3(\Pi) = \HH_3(\Pi)$, and that the former is a quotient of $\HH^0(\Pi)$. 
Indeed, recall that we have already computed the structure of the latter in \cref{prop: center of higher preproj of type A_{2m + 1} times A_{2m + 1}}. 
Hence, we see that its Hilbert series is given by $h_{\HH^0(\Pi)}(x) = \sum_{i = 0}^{\ell - 1}x^{i}$.
By using that $\stHH^0(\Pi) \cong \stHH_3(\Pi)\langle -1 \rangle$ due to \cref{prop: duality formulas}, we deduce that $\overline{\HH}_3(\Pi) = X_2$ is at most one dimensional in degrees $1$ through $\ell$, so that $X_1$ must be trivial.

With this, we know the Hilbert series of $\HH_i(\Pi)$ for $i = 0,1, \ldots, 8$, and thus also for all $i \geq 0$ by using that  $\HH_{i + 8}(\Pi) \cong \HH_{i}(\Pi)\langle 2\ell \rangle$ for $i \geq 1$.
Thus we are also able to give a similar description of the reduced cyclic homology of $\Pi$.

We record what we have just shown in the following theorem, where we recall that $K\mathbb{A}_n$ is $\ell$-homogeneous for $\ell = (n+1)/2$ if $n$ is odd and $\mathbb{A}_n$ is oriented symmetrically.

\begin{thm}\label{thm: cychlic and hochschild homology of higher preproj of type A x A}
Let $\Pi = \Pi_{2 + 1}(K\mathbb{A}_{2\ell - 1}^{\otimes 2})$ with $\ell \geq 2$ and $\mathbb{A}_{2\ell - 1}$ oriented symmetrically, and let $\Pi$ be endowed with the higher preprojective grading.
Then the cyclic homology of $\Pi$ is given by 
\[
\begin{array}{lcl}
\HC_0(\Pi) \cong \Pi/\rad \Pi,
 & \qquad &
\HC_1(\Pi) \cong 0,\\
\HC_2(\Pi) \cong X_2,
& & 
\HC_3(\Pi) \cong 0,\\
\HC_4(\Pi) \cong D(X_2)\langle 2\ell\rangle,
& & 
\HC_5(\Pi)  \cong 0,\\
\HC_6(\Pi)  \cong 0,
& &
\HC_7(\Pi) \cong L\langle 2\ell \rangle,\\
\HC_8(\Pi) \cong 0,
& &
\HC_{i + 8}(\Pi) \cong \HC_{i}(\Pi)\langle  2\ell  \rangle, \quad i \geq 1,
\end{array}
\]
where $h_{\Pi/\rad \Pi}(x) = (2\ell - 1)^2$, $h_{X_2}(x) = \sum_{i = 1}^{\ell}x^{i}$ and $h_{L}(x) = \frac{(2\ell - 2)^2}{2}$. 

Hence, the Hochschild homology of $\Pi$ is then given by 
\[
\begin{array}{lcl}
\HH_0(\Pi)  \cong \Pi/\rad \Pi, & \qquad &
\HH_1(\Pi)  \cong 0, \\
\HH_2(\Pi)  \cong X_2, & & \HH_3(\Pi)  \cong X_2,\\
\HH_4(\Pi)  \cong D(X_2)\langle 2\ell\rangle, & &
\HH_5(\Pi)  \cong D(X_2)\langle 2\ell\rangle,\\
\HH_6(\Pi)  \cong 0, & &
\HH_7(\Pi)  \cong L\langle 2\ell \rangle, \\
\HH_{8}(\Pi)  \cong L\langle 2\ell \rangle, & & \HH_{i + 8}(\Pi) \cong \HH_{i}(\Pi)\langle 2\ell \rangle, \quad i \geq 1.\\
\end{array}
\]
\end{thm}

At this point, we can also use duality formulas to obtain a corresponding description for its Hochschild cohomology. 
Namely, by \cref{prop: duality formulas} we have that
\[
\stHH^{i,j}(\Pi) \cong \stHH_{3 - i, j + 1}(\Pi) 
\]
and 
\[
\stHH^{i,j}(\Pi) \cong D\stHH^{7 - i,-j - 2}(\Pi)
\]
holds for $\Pi$.
Hence, the first formula yields  
$$\HH^{i}(\Pi)  \cong \HH_{3-i}(\Pi)\langle -1 \rangle$$ for $1 \leq i \leq 2$; and
$$\HH^{i}(\Pi) \cong \HH_{11 - i}(\Pi)\langle -2\ell - 1 \rangle$$
for $3 \leq i \leq 10$.
Moreover, by the second formula we have
\[
\HH^{i}(\Lambda) \cong D(\HH^{7 - i}(\Lambda))\langle -2 \rangle;
\]
for $1 \leq i \leq 6$; and 
\[
\HH^{i}(\Lambda) \cong D(\HH^{15 - i}(\Lambda))\langle -2\ell -2 \rangle
\]
for $7 \leq i \leq 14$.

By using these in combination with the preceding theorem, we obtain the following. 
\begin{thm}\label{thm: hochschild cohomology of higher preproj of type A x A}
Let $\Pi = \Pi_{2+1}(K\mathbb{A}_{2\ell -1}^{\otimes 2})$ with $\ell \geq 2$ and $\mathbb{A}_{2\ell - 1}$ oriented symmetrically, and let $\Pi$ be endowed with the higher preprojective grading.
Then the Hochschild cohomology of $\Pi$ is given by
\[
\begin{array}{lcl}
\HH^0(\Pi)  \cong Z, 
& \qquad &
\HH^1(\Pi)  \cong X_2\langle - 1\rangle, \\
\HH^2(\Pi)  \cong 0, 
& & 
\HH^3(\Pi)  \cong L\langle - 1\rangle,\\
\HH^4(\Pi)  \cong L\langle - 1\rangle, 
& &
\HH^5(\Pi)  \cong 0,\\
\HH^6(\Pi)  \cong D(X_2)\langle -1 \rangle, 
& &
\HH^7(\Pi)  \cong D(X_2)\langle -1 \rangle, \\
\HH^{8}(\Pi)  \cong  X_2\langle -2\ell - 1\rangle,
& & 
\HH^{i + 8}(\Pi) \cong \HH^{i}(\Pi)\langle 2\ell \rangle, \quad i \geq 1,\\
\end{array}
\]
where $h_{Z}(x) = \sum_{i = 0}^{\ell- 1}x^{i}$, $h_{X_2}(x) = \sum_{i = 1}^{\ell-1}x^{i}$, and $h_{L}(x) = 2(\ell - 1)^2 $.
\end{thm}

\section*{Acknowledgments}
The authors thank {\O}yvind Solberg for helpful comments on an earlier draft of this paper.

The authors benefited from the use of the software package QPA \cite{QPA} for the GAP system \cite{GAP4} to perform various computations that motivated and informed many of the results in this paper.

The second named author is grateful to have been supported by Norwegian Research Council project 301375, ``Applications of reduction techniques and computations in representation theory'' and by Norwegian Research Council project 357034, ``Symmetries of derived module categories via higher $\tau$-tilting- and almost $T$-Koszul theory'' during work on this paper.

\bibliographystyle{alpha}

\bibliography{big_awful_bib_file.bib}

\appendix
\section{Preliminaries on Dirichlet convolution} \label{sec:dirichlet convolution}

In \cref{sec:product formula for the euler characteristic}, we used several number-theoretic concepts related to arithmetic functions and Dirichlet convolution. Though these concepts are elementary, we review them in this appendix for the benefit of the reader.

An arithmetic function is a function whose domain is the set of positive integers and whose codomain is (a subset of) the complex numbers. Given two arithmetic functions $F$ and $G$, their Dirichlet convolution is the arithmetic function $F * G$ defined by
$$
(F * G)(n) = \sum_{d | n} F(d) G\left(\frac{n}{d}\right).
$$
It may be shown that the set of arithmetic functions forms a commutative ring, with multiplication given by Dirichlet convolution and addition given by pointwise addition, and with the multiplicative identity being the function $\varepsilon$ given by
$$
\varepsilon(n) = \begin{cases}
    1 & \text{if $n = 1$} \\
    0 & \text{otherwise.}
\end{cases}
$$
See e.g.\ Theorem 2.6 and Theorem 2.7 in \cite{apostol-numbertheory} for proofs of some of the ring axioms.

An arithmetic function $F$ is said to be multiplicative if it is not identically zero and $F(mn) = F(m) F(n)$ for all relatively prime pairs of positive integers $m$ and $n$.
The Dirichlet convolution of two multiplicative functions is again a multiplicative function (see e.g.\ \cite[Theorem 2.14]{apostol-numbertheory}).

An important example of a multiplicative arithmetic function is the M\"{o}bius function $\mu$, which is defined by the following formula.
$$
\mu(n) = \begin{cases}
    1 & \text{if $n = 1$} \\
    (-1)^k & \text{if $n$ is the product of $k$ distinct primes} \\
    0 & \text{otherwise.}
\end{cases}
$$
It may be shown that $\mu$ is the inverse with respect to Dirichlet convolution of the constant function $1$ (see e.g.\ \cite[Theorem 2.1]{apostol-numbertheory}). In particular, it follows that if $F$ and $G$ are arithmetic functions such that $G = 1 * F$, then we have $F = \mu * G$. More explicitly, if
$$
G(n) = \sum_{d | n} F(d)
$$
for all $n$, then
$$
F(n) = \sum_{d | n} \mu(d) G\left( \frac{n}{d} \right).
$$
This is known as the M\"{o}bius inversion formula.

The following result is in some sense an analogue of M\"{o}bius inversion for formal power series. While it is used implicitly in \cite{IGUSA1992101} and some version of it should be well known, we have not found a suitable reference  and have thus included a proof.

\begin{lem}\label{lem:convolution and inversion for power series}
Suppose $p(x)$ and $q(x)$ are formal power series over the complex numbers with no constant term, and suppose that $F$ and $G$ are arithmetic functions that are each other's inverses with respect to Dirichlet convolution. Then we have
$$
q(x) = \sum_{k = 1}^\infty F(k) p(x^k)
$$
if and only if
$$
p(x) = \sum_{k = 1}^\infty G(k) q(x^k).
$$
\end{lem}
\begin{proof}
We prove the ``only if'' part of the statement. Suppose that $q(x) = \sum_{k= 1}^\infty F(k) p(x^k)$. Since the power series $q(x)$ has no constant term, the infinite sum $\sum_{k = 1}^\infty G(k) q(x^k)$ is well defined, and we compute its value as
\begin{align*}
\sum_{k = 1}^\infty G(k) q(x^k) &= \sum_{k = 1}^\infty G(k) \sum_{l = 1}^\infty F(l) p(x^{kl}) \\
&= \sum_{n = 1}^\infty \sum_{\substack{k, l \\ kl = n}} F(l) G(k) p(x^n) \\
&= \sum_{n = 1}^\infty \sum_{l | n} F(l) G\left(\frac{n}{l}\right) p(x^n) \\
&= \sum_{n = 1}^\infty (F * G)(n)p(x^n) \\
&= \sum_{n = 1}^\infty \varepsilon(n) p(x^n),
\end{align*}
where the last equality uses the fact that $F$ and $G$ are Dirichlet inverses. Now by using the fact that $\varepsilon(n)$ is $1$ for $n = 1$ and $0$ for all other $n$, we are left with $p(x)$, as desired.
\end{proof}
\end{document}